\documentclass[arxiv]{imsart}

\usepackage{amsmath}
\usepackage{graphicx}
\usepackage{enumerate}
\usepackage{url} % not crucial - just used below for the URL

\usepackage{float}  
\usepackage{subfigure}  
\usepackage[utf8]{inputenc} % allow utf-8 input
\usepackage[T1]{fontenc}    % use 8-bit T1 fonts
\usepackage{hyperref}       % hyperlinks
\usepackage{url}            % simple URL typesetting
\usepackage{booktabs}       % professional-quality tables
\usepackage{amsfonts}       % blackboard math symbols
\usepackage{nicefrac}       % compact symbols for 1/2, etc.
\usepackage{microtype}      % microtypography
\usepackage{lipsum}		% Can be removed after putting your text content
\usepackage{amsthm}
\usepackage{appendix}
\usepackage{mathrsfs}
\usepackage{amssymb}
\usepackage{color}
\usepackage{enumitem}
\usepackage{algorithm}  
\usepackage{algorithmicx}  
\usepackage[noend]{algpseudocode} 
\usepackage{epsfig}
\usepackage{comment}

\usepackage{algorithm}  
\usepackage{algorithmicx}  
\usepackage[noend]{algpseudocode}  
\usepackage{amssymb}
\usepackage{bbm}
\usepackage{comment}
\usepackage{amsthm}
\usepackage{multirow}
\usepackage{bm}
\usepackage{amsmath}
\usepackage{amsfonts}
\usepackage[dvipsnames]{xcolor}
\RequirePackage[authoryear]{natbib}%% uncomment this for author-year citations
\RequirePackage[colorlinks,citecolor=blue,urlcolor=blue]{hyperref}

\startlocaldefs
\theoremstyle{plain}

\newtheorem{theorem}{Theorem}[section]
\newtheorem{lemma}[theorem]{Lemma}
\newtheorem{corol}[theorem]{Corollary}
\newtheorem{proposition}[theorem]{Proposition}
\theoremstyle{remark}

\newtheorem*{remark}{Remark}
\newtheorem{assumption}{Assumption}

\newcommand{\field}[1]{\mathbb{#1}}
\newcommand{\R}{\field{R}}
\newcommand{\E}{\field{E}}
\newcommand{\F}{\mathcal{F}}
\newcommand{\cL}{\mathcal{L}}
\newcommand{\cH}{\mathcal{H}}

\newcommand{\N}{\field{N}}
\newcommand{\Z}{\field{Z}}

\def\qed{\hfill$\diamondsuit$}

\def\cod{\stackrel{\cal D}{\longrightarrow}}
\def\cop{\stackrel{\cal P}{\longrightarrow}}

\def\vline{\hfil\kern\arraycolsep\vline\kern-\arraycolsep\hfilneg}
\def\th{\tilde h}
\def\thl{\th_l}
\def\tsig{\tilde\Sigma}
\def\hsig{\hat\Sigma}

\def\sgn{{\mbox{sgn}}}
\def\cov{{\mbox{cov}}}

\def\var{{\mbox{var}}}
\def\tr{{\mbox{tr}}}

\def\cum{{\mbox{cum}}}
\def\bz{\bm{0}}
\def\deltaeq{\overset{\Delta}{=}}
\def\bl{\bm{l}}
\def\balpha{\bm{\alpha}}
\def\bbeta{\bm{\beta}}
\def\iid{\stackrel{i.i.d.}{\sim}}

\endlocaldefs

\begin{document}
\title{Adaptive Testing for High-Dimensional Data}
%\title{A sample article title with some additional note\thanksref{t1}}
\runtitle{Adaptive Testing for High-Dimensional Data}

\begin{frontmatter}

\begin{aug}
%%%%%%%%%%%%%%%%%%%%%%%%%%%%%%%%%%%%%%%%%%%%%%%
%% Only one address is permitted per author. %%
%% Only division, organization and e-mail is %%
%% included in the address.                  %%
%% Additional information can be included in %%
\author[A]{\fnms{Yangfan}~\snm{Zhang}\ead[label=e1]{yangfan3@illinois.edu}},
\author[B]{\fnms{Runmin}~\snm{Wang}\ead[label=e2]{runminw@tamu.edu}}
\and
\author[A]{\fnms{Xiaofeng}~\snm{Shao}\ead[label=e3]{xshao@illinois.edu}}
%%%%%%%%%%%%%%%%%%%%%%%%%%%%%%%%%%%%%%%%%%%%%%
%% Addresses                                %%
%%%%%%%%%%%%%%%%%%%%%%%%%%%%%%%%%%%%%%%%%%%%%%
\address[A]{Department of Statistics,
University of Illinois at Urbana Champaign\printead[presep={,\ }]{e1,e3}}

\address[B]{Department of Statistics,
Texas A\&M University\printead[presep={,\ }]{e2}}
\end{aug}

\begin{abstract}
In this article, we propose a class of $L_q$-norm based U-statistics for a family of global testing problems related to high-dimensional data. This includes testing of mean vector and its spatial sign,  simultaneous testing of linear model coefficients, and testing of component-wise independence for high-dimensional observations, among others. Under the null hypothesis, we derive asymptotic normality and independence between $L_q$-norm based U-statistics for several $q$s  under mild moment and cumulant conditions. A simple combination of two studentized $L_q$-based test statistics via their $p$-values is proposed and is shown to attain great power against alternatives of different sparsity. Our work is a substantial extension of \cite{he2021asymptotically}, which is mostly focused on mean and covariance testing, and we manage to provide a general treatment of asymptotic independence of $L_q$-norm based U-statistics for a wide class of kernels. To alleviate the computation burden, we introduce a variant of the proposed U-statistics by using the monotone indices in the summation, resulting in a U-statistic with asymmetric kernel. A dynamic programming method is introduced to reduce the computational cost from $O(n^{qr})$, which is required for the calculation of the full U-statistic, to $O(n^r)$ where $r$ is the order of the kernel. Numerical studies further corroborate the advantage of the proposed adaptive test as compared to some existing competitors. 
\end{abstract}
\end{frontmatter}

\section{Introduction}
High-dimensional data, where the dimension $p$ could be comparable to or exceeds the sample size $n$,   are ubiquitous and are encountered on a regular basis in many scientific fields. In practice, it is often of interest to
test  some overall patterns of low-dimensional features of high-dimensional data. This includes testing of mean vectors \citep{chenqin2010,cai2014two,wang2015high,he2021asymptotically}, covariance matrices \citep{li2012two,he2021asymptotically}, and regression coefficients in the linear models \citep{zhong2011tests}, as well as component-wise independence in high-dimensional observations \citep{han2017distribution,leung2018testing}.

These problems can be formulated as $H_0:\Theta=\bm{0}$, where $\bm{0}$ is a vector of all zeros,
$\Theta=\{\theta_l, l\in {\cal L}\}$ is a high-dimensional vector with ${\cal L}$ being the index set, and $\{\theta_l\}$ being the corresponding parameters of interest. For this type of global testing problem, many methods have been developed  in the literature, and there are two dominating brands. One is the $L_2$ norm based test with sum-of-squares-type statistics, which is widely used in the context of mean testing \citep{bai1996effect,chenqin2010,goeman2006testing,gregory2015two,srivastava2008test,srivastava2016raptt}, covariance testing \citep{bai2009corrections,chen2010tests,ledoit2002some,li2012two}, and other testing problems such as component-wise independence test \citep{leung2018testing}. The other is the $L_{\infty}$ norm based test with maximum-type statistics. See examples in mean testing \citep{cai2014two,hall2010innovated}, covariance testing \citep{cai2011limiting,cai2013two,jiang2004asymptotic,liu2008asymptotic,shao2014necessary} and component-wise independence testing \citep{han2017distribution,drton2020high}. It is well known that sum-of-squares-type statistics are powerful against dense alternative, where $\Theta$ has a large proportion of nonzero elements with large $\|\Theta\|_2$, whereas maximum-type statistics target at sparse alternative, where $\Theta$ has few nonzero elements with a large $\|\Theta\|_{\infty}$. In practice, it is often unrealistic to assume a particular type of alternative and there is little knowledge about the type of alternative if any.  Thus there is a need to develop new test that can be adaptive to different types of alternatives, and have good power against a broad range of alternatives, including
moderately sparse and moderately dense ones. 

In addition to $L_2$ and $L_\infty$ type statistics, \cite{xu2016adaptive} and \cite{wu2019adaptive} developed a family of von Mises V-statistics of $L_q$-type for testing $\|\Theta\|_q^q$ for $q=2,3,\cdots$, which are usually biased estimators. In a recent paper, \cite{he2021asymptotically} proposed a U-statistics framework and constructed unbiased and asymptotically independent estimators of $\|\Theta\|_q^q$ .
Then they combine the $L_q$-norm based test statistics, including the cases $q=2,6,\infty$,   by aggregating the individual $p$-values. The resulting test is adaptive in the sense that it can capture a wide range of alternatives. However, the U-statistics considered in \cite{he2021asymptotically} mainly focus on the mean and covariance testing problem and its corresponding kernel function is of order 1. It is natural to ask whether the asymptotic normality and independence for $L_q$-norm based U-statistics hold for more general kernels and can go beyond the mean and covariance testing problem. 

In this article, we aim to advance the $L_q$-norm based tests and their combination to a broader set of testing problems.
In particular, our framework provides a more general treatment and includes U-statistic of order larger than 1, which naturally occurs in problems including testing the nullity of linear regression coefficients [\cite{zhong2011tests}], and component-wise independence testing based on Kendall's $\tau$ and Spearman's $\rho$ \citep{leung2018testing}. In addition, our framework allows us to extend the spatial sign-based test developed in \cite{wang2015high} for $L_2$ to its $L_q$ counterpart, which seems not included in \cite{he2021asymptotically}. Our asymptotic theory covers some of those developed in \cite{he2021asymptotically} as special cases, and further broadens the applicability of the $L_q$ norm based tests and their combination. In our theory development, we also derive a new Hoeffding decomposition that is finer than the classical one, by taking into account the special structure of the kernel functions. We note that the new decomposition holds for a wider class of kernels with tensor product structures and can be of broad interest.

Computationally speaking, the calculation of full U-statistics for $\|\Theta\|_q^q$, $q=4,6,\cdots$ is very challenging. He et al. (2021) developed a computationally efficient algorithm to make the computation feasible by taking advantage of a recursive structure in the statistics. However, their algorithm seems only applicable to the kernel of order 1, and an extension to the higher order case is difficult. To this end, we propose a family of U-statistics with
monotone index in the summation to ease the computational burden of calculating the full U-statistics. They are
essentially U-statistics with asymmetric kernels, which can be computed much faster than the full U-statistics via dynamic programming. While we pay a price in terms of  a constant loss of statistical efficiency, the new U-statistics with asymmetric kernels can be calculated efficiently in $O(n^r)$ time, as opposed to $O(n^{qr})$ for the brute-force approach,  where $r$ equals the order of the kernel function.

The rest of the paper is organized as follows. In Section \ref{sec::background}, we introduce the problem settings and propose the family of $L_q$-type U-statistics for both one-sample and two-sample testing problems. We also include several motivating examples where the proposed statistics can be applied. Section \ref{sec::theory} includes the main theory, such as  asymptotic normality and independence of the proposed U-statistics with different $q$'s under the null. The power analysis under the alternative is also included. In Section \ref{sec::dpustat}, we propose a variant of U-statistic with asymmetric kernels, for which dynamic programming can be applied to compute  the statistic efficiently. Section \ref{sec::sim} includes all simulation results, and Section \ref{sec:data} includes a real data application. We conclude the paper in Section \ref{sec:conclude}. All technical  proofs are gathered in the supplement.

\textbf{Notations.} We use capital letters (e.g. $X_i$) to denote random vectors, and corresponding lowercase ones (e.g. $x_{ij}$) to denote their components. We use $\sum^*$ to denote the summation over all mutually distinct indices.  For two functions $f(X_1,\ldots,X_r)$ and $g(Y_1,\ldots,Y_s)$, we define their tensor product $f\otimes g$ as functions
$$
(f\otimes g)(X_1,\ldots,X_r;Y_1,\ldots,Y_s)=f(X_1,\ldots,X_r)g(Y_1,\ldots,Y_s).
$$
We define $(\otimes)^s f=f\otimes [(\otimes)^{s-1}f]$ for $s\ge 2$ recursively. 
We denote $\{1,\ldots,n\}$ as $[n]$, and $\{m,m+1,\ldots, n\}$ as $[m,n]$. We use $I$ to denote a sequence of (not necessarily distinct) indices in $[n]$, $(i_1,i_2,\ldots,i_m)$, and we denote $(X_{i_1},X_{i_2},\ldots,X_{i_m})$ as $X_I$. In particular, we have $X_{[c]}=(X_1,\ldots,X_c)$. We define $P_m(n)$ as the collection of all subsets of $[n]$ with cardinality $m$, i.e.
$$
P_m(n)=\{\{i_1,i_2,\ldots,i_m\}\subseteq [n]: \text{ All indices are distinct with each other}\}.
$$
For $f(n)$ and $g(n)$, we write $f\asymp g$ if $f(n)/g(n) \rightarrow 1$ as $n\rightarrow\infty$.
Throughout the paper, we define $\|X\|_q^q=\sum_{l=1}^px_l^q$ for any vector $X=(x_1,\ldots,x_p)\in\R^p$ and $q\in \Z_+$. Note that we did not take the absolute value of $x_l$ and hence $\|\cdot\|_q$ is not equal to the classical $l_q$ norm for odd $q$. We use $e_t(r)$ to denote a vector in $\R^r$ with $t$-th element being 1 and others 0.

\section{Background}
\label{sec::background}
\subsection{One-Sample Test}
Suppose we observe some i.i.d. (independent and identically distributed) samples  $X_1,\ldots,X_n\in\R^p$ from some underlying distribution and $p$ is comparable to or exceeds sample size $n$. The goal is to test some high-dimensional parameters $\Theta=\Theta(X_i)=\{\theta_l:l\in\cL\}=\bz$. We can construct a family of unbiased U-statistics for $\|\Theta\|_q^q=\sum_{l\in\cL}\theta_l^q$. For each $l$, we start with some kernel functions $h_l$ of order $r$ satisfying $\E[h_l(X_{[r]})]=\theta_l$. We can then introduce a new kernel $\otimes^q h_l$ with order $qr$, whose expectation equals $\theta_l^q$. Recall that
$$\otimes ^q h_l(X_{[qr]})=h_l(X_{[r]})h_l(X_{[r+1,2r]})\cdots h_l(X_{[(q-1)r+1,qr]}).$$
Let $P^n_m=\frac{n!}{(n-m)!}$. Then we can define the associated U-statistic
$$
U_{n,q}=\sum_{l\in\cL}(P^n_{qr})^{-1}\sum^*_{1\le i_1,\ldots,i_{qr}\le n}\prod_{c=1}^q h_l(X_{i_{(c-1)r+1}},\ldots,X_{i_{cr}}).
$$
In the definition of $L_q$ norm above, we do not take the absolute value (in contrast to classical definition), due to the construction of the U-statistic. In the extreme case, where the positive components cancel out with the negative ones, the $L_q$ norm $\|\Theta\|_q^q$ can be (nearly) zero for odd $q$, even if $\Theta$ is not a zero vector. Therefore, we will mainly consider even $q$ in this paper, i.e. $q\in2 \Z_+$, although our theoretical results can be extended to general $q\in \Z_+$. 

Note that the one-sample mean testing problem studied in \cite{he2021asymptotically} is included in the proposed framework. Specifically,  $\cH_0: \E[X_i]= \mu_0=(\mu_{0,1},\ldots,\mu_{0,p}$), the family of U-statistics proposed in \cite{he2021asymptotically} can be recovered by taking $\cL=[p]$, order $r=1$ and $h_l(X_i)=x_{i,l}-\mu_{0,l}$. We have $\E[U_{n,q}]=\|\E[X_1]-\mu_0\|_q^q$. For mean-testing statistics, the asymptotic independence for $\{U_{n,q}\}_{q\in Z^+}$ is derived from asymptotic joint normality along with $\cov(U_{n,q_1},U_{n,q_2})=0$ for $q_1\not=q_2$. For illustration, consider the example where $q_1>q_2$. For any summand $X_{i_1}\cdots X_{i_{q_1}}$ in $U_{n,q_1}$, and $X_{j_1}\cdots X_{j_{q_2}}$ in $U_{n,q_2}$, there must be at least one $c$ such that $i_c$ is not in the set $\{j_1,\ldots,j_{q_2}\}$ and hence their covariance will be 0. This argument highly relies on the fact that the kernel function has order 1. For kernels with higher order, we cannot expect the covariance to be exactly 0, but we may still show the asymptotic negligibility of the covariance  (after proper scaling) and hence asymptotic independence between these U-statistics for different $q$'s, by imposing some regularity assumptions.
%moderate constraints on the cumulants of the kernel functions.

\subsection{Two-Sample Test}
Suppose we observe two i.i.d. samples $X_1,\ldots,X_n\in\R^p$ and $Y_1,\ldots,Y_m\in\R^p$. And we want to test some high-dimensional parameter $\Theta=\Theta(X_i,Y_i)=\{\theta_l:l\in\cL\}=\bz$. The approach is similar to one-sample test, as we are now interested in $\|\Theta\|_q^q$ and we want to test if $\|\Theta\|_q^q=0$.

 For each $l\in\cL$, we can construct a family of two-sample U-statistic based on some kernel function $h_l(X_{[r]};Y_{[r]})$ of order $(r,r)$ satisfying $\E[h_l(X_{[r]};Y_{[r]})]=\theta_l$. Similar to one-sample testing, we may consider $\otimes^qh_l$ with order $(qr,qr)$ whose expectation equals $\theta_l^q$, i.e.
$$\otimes^qh_l(X_{[qr]};Y_{[qr]})=h_l(X_{[r]};Y_{[r]})\cdots h_l(X_{[(q-1)r+1,qr]};Y_{[(q-1)r+1,qr]}).$$

We can then define the associated U-statistic
\begin{align*}
    U_{n,m,q}=&\sum_{l\in\cL}(P^n_{qr}P^m_{qr})^{-1}
    \sum^*_{1\le i_1,\ldots,i_{qr}\le n}\sum^*_{1\le j_1,\ldots,j_{qr}\le m}\\
    &\prod_{c=1}^q
    h_l(X_{i_{(c-1)r+1}},\ldots,X_{i_{cr}};Y_{i_{(c-1)r+1}},\ldots,Y_{i_{cr}}).
\end{align*}

As is the case for one-sample testing problem, the family of statistics for two-sample mean testing in \cite{he2021asymptotically} is also a special case of the above two-sample U-statistics.

\subsection{Motivating Examples}
\label{sec::eg}
In this section, we present some examples that are included in the above formulation. We mainly focus on the examples in one-sample testing problems, as their counterparts for two-sample testing can be constructed in a straightforward way.

\begin{itemize}
    \item Mean (zero) testing, $\cH_0:\mu:=\E[X_i]=\bm{0}$ \citep{he2021asymptotically}. In this case, $\cL=[p]$, $r= 1$ and $h_l(X_i)=x_{i,l}$. 
    
    \item Spatial sign based testing \citep{wang2015high}, $\cH_0:\E[X_i\,/\,\|X_i\|\,]=\bm{0}$. We can take $\cL=[p]$, $r= 1$ and $h_l(X_i)=x_{i,l}\,/\,\|X_i\|$.  
     \item Covariance testing for data with zero mean, $\cH_0:\Sigma:=\var(X_i)=\Sigma_0$ for some given $\Sigma_0=(\sigma_{p_1p_2})_{p_1,p_2=1}^p$.  In this case, $\cL=\{l=(p_1,p_2),1\le p_1\le p_2\le p\},r= 1$ and $h_l(X_i)=x_{ip_1}x_{ip_2}-\sigma_{p_1p_2}$. In particular, to test if the covariance has banded structure, i.e. $\sigma_{p_1,p_2}=0,|p_1-p_2|\ge d$, we can take $\cL=\{l=(p_1,p_2)\in[p]^2, p_2\ge p_1+d\},r= 1$ and $h_l(X_i)=x_{ip_1}x_{ip_2}$. This example is also studied in \cite{he2021asymptotically}.

\item Component-wise independence testing. $\cH_0:x_{i,p_1}$ and $x_{i,p_2}$ are independent for any $p_1\not=p_2$. For Kendall's $\tau$ based test \citep{leung2018testing}, we essentially test $\cH_0^\tau:P(x_{1,p_1}<x_{2,p_1},x_{1,p_2}<x_{2,p_2})=1/4$ for any $p_1\not=p_2$, against $\cH_1^\tau$: negation of $\cH_0^\tau$. Let $\cL=\{l=(p_1,p_2),1\le p_1< p_2\le p\},r= 2$ and $h_l(X_i,X_j)=\sgn(x_{i,p_1}-x_{j,p_1})\sgn(x_{i,p_2}-x_{j,p_2})$. 

\item Component-wise independence testing with Spearman's $\rho$ \citep{leung2018testing}. In this case, $\cH_0^\rho:P(x_{1,p_1}<x_{2,p_1},x_{1,p_2}<x_{3,p_2})=1/4$ for any $p_1\not=p_2$, and $\cH_1^\rho$ is the negation of $\cH_0^\rho$. We take $\cL=\{l=(p_1,p_2),1\le p_1< p_2\le p\}$, $r=3$ and $h_l(X_i,X_j,X_k)$ to be the symmetric version of $\sgn(x_{i,p_1}-x_{j,p_1})\sgn(x_{i,p_2}-x_{k,p_2})$, i.e.
\begin{align*}
h_l(X_i,X_j,X_k)=&\frac{1}{6}\big[\sgn(x_{i,p_1}-x_{j,p_1})\sgn(x_{i,p_2}-x_{k,p_2})+\sgn(x_{i,p_1}-x_{k,p_1})\sgn(x_{i,p_2}-x_{j,p_2})\\
&+\sgn(x_{j,p_1}-x_{i,p_1})\sgn(x_{j,p_2}-x_{k,p_2})+\sgn(x_{j,p_1}-x_{k,p_1})\sgn(x_{j,p_2}-x_{i,p_2})\\
&+\sgn(x_{k,p_1}-x_{i,p_1})\sgn(x_{k,p_2}-x_{j,p_2})+\sgn(x_{k,p_1}-x_{j,p_1})\sgn(x_{k,p_2}-x_{i,p_2})\big].
\end{align*}
\item Simultaneous test of linear regression coefficients \citep{zhong2011tests}. For linear model $Y=X\bbeta+\varepsilon$, we test $\cH_0:\bbeta=\bbeta_0$ against $\cH_a:\bbeta\not=\bbeta_0$. We take $\cL=[p],r=2$, and
\begin{equation}
h_l\big((X_1,Y_1),(X_2,Y_2)\big)=[\,(X_1-X_2)(Y_1-Y_2-(X_1-X_2)^T\bbeta_0)\,\big]_l/2,
\label{lrhl}
\end{equation}
and the parameter of interest equals $\Theta=\Sigma(\bbeta-\bbeta_0)$ with $\Sigma:=\var(X)$.

\item Two-sample spatial sign test \citep{chakraborty2017tests}, $\cH_0:\Theta=\bm{0}$ with $\Theta=\E[\,(X_i-Y_j)\,/\,\|X_i-Y_j\|\,]$. We take $\cL=[p]$, $r=1$ and two-sample kernel
$$
h_l(X_i,Y_j)=(x_{i,l}-y_{j,l})/\,\|X_i-Y_j\|.
$$

\end{itemize}

\section{Theoretical Analysis}
\label{sec::theory}
This section includes the asymptotic theory for the proposed statistics. We start with kernels of order one and then move onto kernels of higher order.

\subsection{Asymptotic Theory for Order-One Kernels}
\label{sec::cltr1}
For kernels of order $r=1$, we have
$$
U_{n,q}=\sum_{l\in\cL}(P^n_{q})^{-1}\sum^*_{1\le i_1,\ldots,i_{q}\le n}\prod_{c=1}^q h_l(X_{i_c}).
$$
Write $\cL=[L]$ and define $Y_i=(y_{i,1},\ldots,y_{i,L})^T=(h_1(X_i),\ldots,h_L(X_i))^T\deltaeq H(X_i)$, with $\Sigma=(\sigma_{l_1,l_2})_{l_1,l_2\in\cL}=\var(Y_i)$. In this way, we can view $U_{n,q}$ as the U-statistic for testing $\E[Y_i]\equiv \bm{0}$. The analysis for mean-testing would directly apply.
For example, we can calculate the variance as 
\begin{align*}
\var(U_{n,q})=&(P^n_{q})^{-2}\sum_{l_1,l_2\in\cL}\sum^*_{1\le i_1,\ldots,i_{q}\le n}\sum^*_{1\le j_1,\ldots,j_{q}\le n}\E\left[\,\prod_{c=1}^q y_{i_c,l_1} y_{j_c,l_2}\right]\\
=&(P^n_{q})^{-2}q!\sum_{l_1,l_2\in\cL}\sum^*_{1\le i_1,\ldots,i_{q}\le n}\prod_{c=1}^q\E[y_{i_c,l_1} y_{i_c,l_2}]\\
\asymp &q!n^{-q}\sum_{l_1,l_2\in\cL}\sigma_{l_1,l_2}^q=q!n^{-q}\|\Sigma\|_q^q.    
\end{align*}
For $q_1\not=q_2$, we have 
\begin{align*}
\cov(U_{n,q_1},U_{n,q_2})=&(P^n_{q_1})^{-1}(P^n_{q_2})^{-1}\sum_{l_1,l_2\in\cL}\sum^*_{1\le i_1,\ldots,i_{q_1}\le n}\sum^*_{1\le j_1,\ldots,j_{q_2}\le n}\prod_{c_1=1}^{q_1}\prod_{c_2=1}^{q_2}\E[ y_{i_{c_1},l_1}  y_{j_{c_2},l_2}]=0,
\end{align*}
since there must be at least one $i_{c_1}$ or $j_{c_2}$ that is distinct from any other index when $q_1\not= q_2$. 

The analysis in \cite{he2021asymptotically} can be applied directly, but instead we shall follow the argument used in \cite{zhang2021adaptive} and impose the following weaker moment conditions and some cumulant assumptions.
\begin{assumption}[$r=1$]
Suppose $Y_1,\ldots,Y_n$ are i.i.d. copies of $Y_0$ with  variance matrix $\Sigma$, and the following conditions hold.
\begin{enumerate}
    \item There exists constants $C_p>0$ depending on $p$ such that $\tr(\Sigma) \geq C_pL.$
    \item $Y_0$ has up to fourth moment with $\sup_{l\in\cL}\E[y_{0,l}^4]<C$, and for $u=2,3,4$ %there exist constants $C_u>0$ depending on $u$ only and a constant $\kappa>2$ such that 
    we have $$\left|\operatorname{cum}\left(y_{0, l_{1}}, \ldots, y_{0, l_u}\right)\right| \leq C_p^{u/2}\left(1 \vee \max _{1 \leq i, j \leq u}\left|l_{i}-l_{j}\right|\right)^{-\kappa},\quad\text{ for some }\kappa>2.$$
\end{enumerate}
\label{ass::cumr}
\end{assumption}
The above assumptions are slightly weaker than those in \cite{zhang2021adaptive} but sufficient for deriving the asymptotic distributions as presented in the following theorem, since we do not need to show the process convergence, as required in \cite{zhang2021adaptive}. The above cumulant assumption is implied by geometric moment contraction [cf. Proposition 2 of \cite{wu2004limit}] or physical dependence measure proposed by \cite{wu2005nonlinear} [cf. Section 4 of \cite{shao2007local}], or $\alpha-$mixing (with polynomial decaying rate for mixing coefficients)  [\cite{andrews1991heteroskedasticity}, \cite{zhurbenko1975higher}] in the time series setting. It basically imposes weak dependence among the $p$ components in the data. Our theory does not require an ordering of components and holds as long as a permutation of $p$ components satisfies the cumulant assumption, as our test statistic is invariant to the permutation of components. 
%i.e. the joint cumulant of components $y_{0,l_1},...,y_{0,l_u}$ decays at a polynomial rate in $\max_{i,j = 1,...,u} {|l_i - l_j|}$. 
\begin{theorem}[$r=1$]
Suppose Assumption \ref{ass::cumr} holds. Under the null, we have for any $q\in2\Z_+$,
$$
(q!)^{-1/2}n^{q/2}\|\Sigma\|_q^{-q/2}U_{n,q}\cod N(0,1).
$$
Furthermore, for any finite set $I\subseteq 2\Z_+$, $(U_{n,q})_{q\in I}$ are asymptotically jointly independent.
\label{thm::mainr1}
\end{theorem}

For mean and covariance testing, Assumption \ref{ass::cumr} usually holds with $C_p$ being a constant independent of $p$. %In general, the flexibility of $C_p$ allows the framework applicable to more problem settings. 
For spatial sign based testing, we have $H(X_i)=X_i/\|X_i\|_2$. In this case, we have $\displaystyle\sum_{l\in\cL}\E[y_{i,l}^2]=\E[\,\|Y_i\|_2^2\,]\equiv 1$, %which implies $\inf_{l\in\cL} \var(y_{0, l})\le1/p\rightarrow 0$ as $p\rightarrow\infty$ 
so Assumption \ref{ass::cumr} does not hold if $C_p$ is a fixed constant as in \cite{zhang2021adaptive}. To this end, we provide a separate proof to show the CLT for spatial sign-based test statistics; see Appendix \ref{sec:verify_spatial} for details. 
%However, they may still be satisfied with some constant $C_p$ depending on $p$. To this end, we refer the readers to Appendix \ref{sec:verify_spatial} for the verification of assumptions for spatial sign testing problem. In conclusion, 
To summarize, for order-one kernels, the proposed U-statistics include those studied in \cite{he2021asymptotically} as a special case, and the asymptotic theory is justified under weaker assumptions (e.g. we only require the existence of fourth-order moment). Moreover, the theory can be extended to spatial sign based test, which is not covered by \cite{he2021asymptotically}. On the other hand, \cite{he2021asymptotically} obtained asymptotic independence among $L_q$-based U-statistics and maxmum-type statistics for mean and covariance testing problems, which is not covered here.

\subsection{Hoeffding Decomposition for Kernels with Higher Order}
For kernels with order $r>1$, the asymptotic theory is more involved. We start with a new Hoeffding decomposition that takes into account the special structure of $U_{n,q}$. Without loss of generality, here we assume $h_l$ is symmetric over its arguments for all $l \in \mathcal{L}$. We also assume $\E[h_l(X_{[r]})]=0$ as we focus on the null.

For $1\le t\le r$, define the projected kernels
\begin{align*}
h_{l,t}(x_1,\ldots,x_t)=&\E[h_l(x_1,\ldots,x_t,X_{t+1},\ldots,X_r)]\\
=&\E[h_l(X_1,\ldots,X_t,X_{t+1},\ldots,X_r)\mid X_1=x_1,\ldots,X_t=x_t].
\end{align*}
We call $h_l$ has degeneracy of order $s\in[1,r]$, if $h_{l,t}\equiv 0$ for any $1\le t<s$ while $h_{l,s}\not=0$. In particular, $h_l$ is called non-degenerate if $s=1<r$; and is fully degenerate if $s=r$. Note that we view $h_l$  with $r=1$ as fully degenerate for convention.

\begin{comment}
\Gamma_s(r,q,c)=\{\balpha=(\alpha_1,\ldots,\alpha_r):\alpha_t=0~\forall t<s,\alpha_{t-1}\le\alpha_t,\sum_{t=1}^r\alpha_t=q,\sum_{t=1}^rt\alpha_t=c\}.
$$
\end{comment}

We further define the following (order-$t$ for $1\le t\le r$) Haj\'{e}k projection
\begin{align*}
h_l^{(t)}(x_1,\ldots, x_t)=&\sum_{s=1}^{t}(-1)^{t-s} \times 
 \sum_{1 \leq i_1<\cdots<i_s \leq t} h_{l,s}(x_{i_{1}},\ldots, x_{i_{s}}).
\end{align*}
In particular, we have $h_l^{(1)}=h_{l,1}$. When $r=2$, we have $$h_l^{(2)}(X_1,X_2)=h_l(X_1,X_2)-h_l^{(1)}(X_1)-h_l^{(1)}(X_2),$$ 
which is a fully degenerate kernel. In general, $h_l^{(t)}$ is always fully degenerate with order $t$.

The Haj\'{e}k projection is the cornerstone for deriving the Hoeffding's decomposition. In fact, the classical Hoeffding decomposition is derived based on the fact that a kernel function equals the summation over all Haj\'{e}k projected kernels of different orders. However, instead of directly applying the projection to $\otimes^q h_l$, we write $h_l$ as the summation over the projected kernels $h_l^{(t)}$ and apply the distribution law to decompose $\otimes^q h_l$ into some products of $h_l^{(t)}$. 

For illustration, consider a kernel $h_l$ with order $r=2$. We may write $$h_l(X_1,X_2)=h_l^{(1)}(X_1)+h_l^{(1)}(X_2)+h_l^{(2)}(X_1,X_2).$$
By applying distribution law, we have, e.g., for $q=2$,
\begin{equation}
\begin{aligned}
&(\otimes^2 h_l)(X_1,X_2,X_3,X_4)\\
=&[h_l^{(1)}(X_1)+h_l^{(1)}(X_2)+h_l^{(2)}(X_1,X_2)][h_l^{(1)}(X_3)+h_l^{(1)}(X_4)+h_l^{(2)}(X_3,X_4)]\\
=&(\otimes^2h_l^{(1)})(X_1,X_3)+(\otimes^2h_l^{(1)})(X_1,X_4)+(\otimes^2h_l^{(1)})(X_2,X_3)+(\otimes^2h_l^{(1)})(X_2,X_4)\\
&+(h_l^{(1)}\otimes h_l^{(2)})(X_1,X_3,X_4)+(h_l^{(1)}\otimes h_l^{(2)})(X_2,X_3,X_4)\\
&+(h_l^{(1)}\otimes h_l^{(2)})(X_3,X_1,X_2)+(h_l^{(1)}\otimes h_l^{(2)})(X_4,X_1,X_2)\\
&+(\otimes^2h_l^{(2)})(X_1,X_2,X_3,X_4).   
\end{aligned}
\label{eq::dist_q2r2}
\end{equation}
Therefore, we may decompose $\otimes^2 h_l$ into $\otimes^2h_l^{(1)}$, $h_l^{(1)}\otimes h_l^{(2)}$ and $\otimes^2h_l^{(2)}$.\\
The argument based on distribution law is applicable to general settings, and we may write $\otimes^q h_l$ as the summation of many tensor products of $h_l^{(t)}$. To characterize these summands, we define
\begin{equation}
h^{\balpha}_l=[\otimes^{\alpha_1}h_l^{(1)}]\otimes[\otimes^{\alpha_2}h_l^{(2)}]\otimes\cdots\otimes[\otimes^{\alpha_r}h_l^{(r)}],  
\label{hlalpha}
\end{equation}
for $\balpha\deltaeq(\alpha_1,\ldots,\alpha_r)$. In particular, $\sum_{t=1}^rt\alpha_t$ is the order of $h_l^{\balpha}$, and $\alpha_t$ is the count of $h_l^{(t)}$ in the formula (\ref{hlalpha}). It is straightforward to see that $\otimes^q h_l$ equals the summation of $h_l^{\balpha}$ for all vectors $\balpha\in\N^r$ s.t. $\sum_{t=1}^r\alpha_t=q$. In particular, $\otimes^2h_l^{(1)}$, $h_l^{(1)}\otimes h_l^{(2)}$ and $\otimes^2h_l^{(2)}$ in (\ref{eq::dist_q2r2}) is associated with $\balpha=(2,0)$, (1,1) and (0,2), respectively.

Although $h_l$ is a symmetric kernel, $h_l^{\balpha}$ is not symmetric in general. Define the symmetric counterpart of $h^{\balpha}_l$ as $\tilde h^{\balpha}_l$. Note that $\tilde h^{\balpha}_l$ is a fully degenerate kernel of order $\sum_{t=1}^r t\alpha_t$, as $h^{(t)}_l$ is fully degenerate. For example, the kernel $h_l^{(1)} \otimes h_l^{(2)}$ in (\ref{eq::dist_q2r2}) associated with $\balpha=(1,1)$ is fully degenerate with order $r=3$. It is asymmetric as $h_l^{\balpha}(X_1,X_2,X_3)=h_l^{(1)}(X_1)h_l^{(2)}(X_2,X_3)$, and its symmetric counterpart is defined through
\begin{align*}
\tilde h_l^{\balpha}(X_1,X_2,X_3)=&\frac{1}{3}\Big[h_l^{(1)}(X_1)h_l^{(2)}(X_2,X_3)+h_l^{(1)}(X_2)h_l^{(2)}(X_1,X_3)+h_l^{(1)}(X_3)h_l^{(2)}(X_1,X_2)\Big].
\end{align*}

We define the following U-statistic associated with $h_l^{\balpha}$, which appears in the Hoeffding decomposition of $U_{n,q}$. Note that $c=\sum_{t=1}^r t\alpha_t$ is the order of $h_l^{\balpha}$.
\begin{align*}
U_{n,\balpha}=&(P_c^n)^{-1}\sum_{l\in\cL}\sum_{1\le i_1,i_2,\cdots,i_c\le n}^*h_l^{\balpha}(X_1,\ldots,X_c)\\
=&\binom{n}{c}^{-1}\sum_{l\in\cL}\sum_{1\le i_1<i_2<\cdots<i_c\le n}\tilde h_l^{\balpha}(X_1,\ldots,X_c).
\end{align*}
With $U_{n,\balpha}$, we have the following decomposition.
\begin{lemma}
There exists some constants $C_{\balpha}>0$ such that
\begin{equation}
U_{n,q}=\sum_{c=qs}^{qr}  \sum_{\balpha:\sum_t t\alpha_t=c}C_{\balpha} U_{n,\balpha},    
\label{eq::hoeff}
\end{equation}
where $s$ is the degeneracy of $h_l$.
\label{lem::hoeff}
\end{lemma}
\begin{remark}
In general, the Hoeffding decomposition (\ref{eq::hoeff}) is a finer composition than the classical one. The classical decomposition is simply based on Haj\'{e}k projection of $\otimes^q h_l$, whereas (\ref{eq::hoeff}) instead takes into account the tensor product structure of $\otimes^q h_l$ and applies distribution law. This finer decomposition allows us to directly impose assumptions on Haj\'{e}k projections of $h_l$ (and some related quantities), otherwise the assumptions will be too abstract to comprehend. We can write the classical Hoeffding decomposition using the kernel $\otimes^qh_l$ as $
U_{n,q}=\sum_{c=qs}^{qr}C^{[c]}U_{n,q}^{[c]},$
where $C^{[c]}$ are some constants, and $U_{n,q}^{[c]}$ is the U-statistic associated with the Haj\'{e}k projection of $\otimes^qh_l$ with order $c$, and includes all $U_{n,\balpha}$ with kernel $\tilde h_l^{\balpha}$ of order $c$.

For example, consider $r=3$ and $q=2$, following the same argument as (\ref{eq::dist_q2r2}), we can decompose $\otimes^3h_l$ into six kernels: $\otimes^2 h_l^{(1)}$, $h_l^{(1)}\otimes h_l^{(2)}$, $h_l^{(1)}\otimes h_l^{(3)}$, $\otimes^2h_l^{(2)}$, $h_l^{(2)}\otimes h_l^{(3)}$ and $\otimes^2h_l^{(3)}$. The classical Hoeffding decomposition, instead, combines the kernels with the same order, e.g. $h_l^{(1)}\otimes h_l^{(3)}$ and $\otimes^2h_l^{(2)}$, and hence only contains 5 terms. It is not difficult to see that the proposed decomposition coincides with the classical one for $r\le 2$.

In general, for any kernel function that can be written as the tensor product of several functions, their Haj\'{e}k projections can be used to derive a finer Hoeffding decomposition in the same way that we obtain (\ref{eq::hoeff}) by applying distribution law.
\end{remark}

It is worth noting that there is one special type of $\balpha$ that plays an important role in the above decomposition, where only $t$-th component of $\balpha$ equals $q$, and all other components are zero, i.e. $\balpha = qe_t(r)$. for $t = 1,...,r$. The associated $h_l^{\balpha}$ is equal to $\otimes^q h_l^{(t)}$. We denote its symmetric version as $H_{l,q}^{(t)}$, so
$$
H_{l,q}^{(t)}(X_{[qt]})=[(qt)!]^{-1}\sum_{I\in P_{qt}(qt)}(\otimes)^q h_l^{(t)}(X_I),
$$ 
and the associated U-statistic as $U_{n,q}^{(qt)}$. In the context of previous illustrating example with $r=2$, $U_{n,q}^{(q)}$ and  $U_{n,q}^{(2q)}$ are the U-statistics associated with $\otimes^qh_l^{(1)}$ and $\otimes^q h_l^{(2)}$ respectively.

In Section \ref{sec::cltr}, we show that the asymptotic variance can be characterized by
\begin{equation}
\tsig_t(q)\deltaeq\sum_{l_1,l_2\in\cL}\cov(H_{l_1,q}^{(t)}(X_{[qt]}),H_{l_2,q}^{(t)}(X_{[qt]}))=\var(H_{q}^{(t)}(X_{[qt]})),
\label{deftsig}
\end{equation}
where we have defined $H_q^{(t)}\deltaeq\sum_{l\in\cL}H_{l,q}^{(t)}$.

Next, we also define $$\sigma_t(l_1,l_2)\deltaeq\cov(h_{l_1}^{(t)}(X_{[t]}),h_{l_2}^{(t)}(X_{[t]})),\quad l_1,l_2\in\cL,$$
and 
\begin{align*}
\Sigma_t(q)\deltaeq\sum_{l_1,l_2\in\cL}\sigma_t^{q}(l_1,l_2)=\sum_{l_1,l_2\in\cL}\cov\big((\otimes^q h_{l_1}^{(t)})(X_{[qt]}),(\otimes^q h_{l_2}^{(t)})(X_{[qt]})\big)
=\var\big(\sum_{l\in\cL}(\otimes^q h_l^{(t)})(X_{[qt]})\big),
\end{align*}
for $q\in 2\Z_+,1\le t\le r$.
The definition is similar to $\tsig_t(q)$, and in particular, we have $\tsig_1(q)=\Sigma_1(q)$ since $\otimes^qh_l^{(1)}$ is symmetric so $H_{l,q}^{(1)}=\otimes^qh_l^{(1)}$. In general, we have $\tsig_t(q)\le \Sigma_t(q)$ due to Cauchy-Schwarz inequality. 
For illustration, consider $t=2$ and $q=2$. We have 
\begin{align*}
H_{2}^{(2)}(X_1,X_2,X_3,X_4)=&\sum_{l\in\cL}\frac{1}{3}\Big[h_l^{(2)}(X_1,X_2)h_l^{(2)}(X_3,X_4)+h_l^{(2)}(X_1,X_3)h_l^{(2)}(X_2,X_4)\\
&+h_l^{(2)}(X_1,X_4)h_l^{(2)}(X_2,X_3)\Big].
\end{align*}
Cauchy-Schwarz inequality then implies that
$$
\cov\left(H_2^{(2)},H_2^{(2)}\right)\le \var\left(\sum_{l\in\cL}\otimes^2h_{l}^{(2)}\right)=\Sigma_2(2),
$$
and hence $\tsig_2(2)\le \Sigma_2(2)$.
We remark that $\Sigma_t(q)$ is more convenient for the order analysis, although it does not characterize the asymptotic variance. 
%We may combine all terms whose kernels have the same order and define 
%$$
%D_{n,q}^{(c)}=\sum_{\balpha\in\Gamma_1(r,q,c)}D_n^{\balpha}.
%$$

%In particular, if $h_l$ is a fully degenerate kernel with order $r$, we can see that $U_{n,q}$ is a special case of above statistic with $\balpha_r(r,q) = (0,\ldots,0,q)\in \R^r$, where $\balpha_r(s,q)$ denotes a vector in $\R^r$ such that only the $s$-th component is non-zero and equal to $q$. 

\subsection{Asymptotic Theory for Order-One Kernels}
\label{sec::cltr}
As the decomposition described in the previous section is a finer version of the traditional Hoeffding decomposition, it inherits most of the properties of  Hoeffding's decomposition such as the orthogonality among projections with different orders. Therefore it can be used to analyze the asymptotic distribution of $U_{n,q}$.  We make the following assumptions to derive the leading term in the decomposition.
\begin{assumption}
Suppose $h_l$ has the degeneracy of order $s$. We have $\Sigma_{t}(q)=o\big(n^{q(t-s)}\Sigma_{s}(q)\big)$ for any $s<t\le r$ and $q\in 2\Z_+$, and $\Sigma_s(q)=O\big(\tsig_s(q)\big)$.
\label{ass::lead_hoeff}
\end{assumption} 
%Recall that $\tsig_t(q)\le \Sigma_t(q)$, so the first statement in the assumption is implied by $\Sigma_t(q)=o\big(n^{q(t-s)}\tsig_s(q)\big)$ and the second implies that $\Sigma_s(q)$ and $\tsig_s(q)$ are of the same order. 
When $s=1$, the assumption is equivalent to $\Sigma_t(q)=o\big(n^{q(t-1)}\Sigma_1(q)\big)$ as $\Sigma_1(q)=\tsig_1(q)$.
For illustration, suppose $h_l$ has order 2. If it is degenerate, then the first statement in the Assumption \ref{ass::lead_hoeff} is automatically satisfied since $s=r=2$ in this case and it suffices to verify the second one. On the other hand, if $h_l$ is not degenerate, i.e. $s=1$, then the second statement in the assumption holds as $\Sigma_1(q)=\tsig_1(q)$ , and it suffices to verify $\Sigma_2(q)=o\big(n^q\Sigma_1(q)\big)$ i.e.
%$$
%\sum_{l_1,l_2\in\cL}\E[H_{l_1,q}^{(2)}(X_{[2q]})H_{l_2,q}^{(2)}(X_{[2q]})]=o\left(n^q\sum_{l_1,l_2\in\cL}\E[H_{l_1,q}^{(1)}(X_{[q]})H_{l_2,q}^{(1)}(X_{[q]})]\right).
%$$
$$
\sum_{l_1,l_2\in\cL}\E^q[h_{l_1}^{(2)}(X_{[2]})h_{l_2}^{(2)}(X_{[2]})]=o\left(n^q\sum_{l_1,l_2\in\cL}\E^q[h_{l_1}^{(1)}(X_1)h_{l_2}^{(1)}(X_1)]\right).
$$
This is trivially satisfied if $\E[h_{l_1}^{(2)}(X_{[2]})h_{l_2}^{(2)}(X_{[2]})]$ and $\E[h_{l_1}^{(1)}(X_1)h_{l_2}^{(1)}(X_1)]$ (and hence the summations) have the same order, which is true in most settings such as simultaneous testing of linear model coefficients and component-wise independence testing discussed in Section \ref{sec::eg}. For example, for Kendall's $\tau$ we have $\Sigma_2(q)=7^q\Sigma_1(q)=(7/9)^qp(p-1)/2$ (see Appendix \ref{sec:verify_kendall} for more details).

In general, when $h_l$ has degeneracy of order $s$, the Haj\'{e}k projection $h_l^{(t)}\equiv 0$ for any $1\le t<s$, which implies $h_{l,\balpha} = 0$ if its order $\sum_t t\alpha_t$ is less than $qs$, and the non-zero kernel of lowest order is $H_{q,l}^{(s)}$ (corresponding to $\balpha=q\cdot e_t(r))$). Note that this is the only non-zero kernel with order $qs$. 
The following lemma states that the U-statistic $U_{n,q}^{(qs)}$ with kernel $H_{q,l}^{(s)}$ is indeed the leading term.
\begin{lemma}
Suppose the kernel $h_l$ has degeneracy of order $s\in[1,r]$. Under Assumption \ref{ass::lead_hoeff}, we have under the null, it holds for any $q\in 2\Z_+$ that $U_{n,q}=\binom{r}{s}^qU_{n,q}^{(qs)}[1+o_p(1)]$
\label{lem::lead_hoeff}
\end{lemma}
In view of Lemma \ref{lem::lead_hoeff}, the leading term $U_{n,q}^{(qs)}$ in the Hoeffding decomposition is a U-statistic that has fully degenerate kernels $\otimes^qh_l^{(s)}$ induced by $h_l^{(s)}$. Therefore it suffices to study the asymptotic theory of U-statistic associated with fully degenerate ($r = s$) kernel $h_l$, without loss of generality, and then apply the result to $h_l^{(s)}$. We make the following assumption for any fully degenerate kernel, which is applicable to $h_l^{(s)}$. Recall that $\sigma_t(l_1,l_2)\deltaeq\cov(h_{l_1}^{(t)}(X_{[t]}),h_{l_2}^{(t)}(X_{[t]}))$.

\begin{assumption}
Suppose $h_l$ is a fully degenerate kernel function with order $r\ge 1$. We have the followings hold for any $q\in 2\Z_+$.\\
(a) ~$\sum_{l_1,l_2,l_3,l_4\in\cL}[\sigma_r(l_1,l_2)\sigma_r(l_3,l_4)\sigma_r(l_1,l_4)\sigma_r(l_2,l_3)]^{q/2}=o\left(\tsig_r^{2}(q)\right)$.\\
(b1) ~
$h_l$ has up to $4$-th moments with $\sup_{l\in\cL}\E[h_l^4(X_{[r]})]< C,$ 
and for any $q$ there exists a constant $C_q$ such that
$$
\sum_{l_1,...,l_u\in\cL}|\cum^q(h_{l_1}(X_{i_1^{(1)}},\ldots,X_{i_r^{(1)}}),\ldots,h_{l_u}(X_{i_1^{(u)}},\cdots,X_{i_r^{(u)}}))|\le C_q\tsig_r^{u/2}(q),
$$
for $u=2,3,4$ and any $u$ groups of distinct indices $(i_1^{(u)},\ldots,i_{r}^{(u)})$.\\
Moreover, if $r>1$, the following holds.\\
(b2) ~
$h_l$ has up to $q$-th moments with $\sup_{l\in\cL}\E[h_l^q(X_{[r]})]< C,$  
and 
$$
\sum_{l_1,...,l_u\in\cL}|\cum^{q/c}(h_{l_1}(X_{I_1}),\ldots,h_{l_u}(X_{I_c}))|=o(\tsig_r^{2}(q)),
$$
for $u=3,4;1\le c\le q$ and $I_1,\ldots,I_c\in P_{cr}(n)$ such that any element only appears in two $I_t$'s and $I_t\not=I_s$ for any $t\not=s$.
\label{ass::asy_deg}
\end{assumption}

\begin{remark}
In summary, Assumption \ref{ass::asy_deg} requires certain summability on the joint cumulants of $H_{l_1,q}^{(s)}, H_{l_2,q}^{(s)}, H_{l_3,q}^{(s)}, H_{l_4,q}^{(s)}$ for all $l_1,...,l_4 \in \mathcal{L}$. In particular, Assumption \ref{ass::asy_deg} (b2) is a new assumption made for $r>1$, compared to those for $r=1$. This is due to the complexity in the fourth moment of kernels with $r>1$. These summability assumptions guarantee that the asymptotic variance $U_{n,q}^{(qr)}$ is indeed characterized by $\tsig_r(q)$ and martingale CLT holds. In fact, Assumption \ref{ass::cumr} implies Assumption \ref{ass::asy_deg} (a) and (b1), and they are natural generalizations of the assumptions made in \cite{zhang2021adaptive}, whose argument is focused on $r=1$.
%although they derive stronger process convergence. 

Assumption \ref{ass::asy_deg} represents abstract properties of the $h_l$ since our framework allows general kernels and the verification for them requires case-by-case analysis. To shed some light on its feasibility, we demonstrate the verifiability of our assumptions using several commonly studied U-statistics and the assumptions can all be straightforwardly verified. We include the results in Appendix \ref{sec:verify}. 
\end{remark}
\begin{proposition}[Degenerate kernel, $r=s$]
Suppose Assumption \ref{ass::asy_deg} holds. Under the null, we have for any $q\in 2\Z_+$,
$$
[(qr)!]^{-1/2}n^{qr/2}\tsig_r^{-1/2}(q)U_{n,q}\cod N(0,1).
$$
Furthermore, for any finite set $I\subseteq 2\Z_+$, $(U_{n,q})_{q\in I}$ are asymptotically jointly independent.
\label{thm::mainr_deg}
\end{proposition}
Combining Proposition \ref{thm::mainr_deg} and Lemma \ref{lem::lead_hoeff}, we immediately have the following result.
%\begin{lemma}
%Assumptions implies Assumption \ref{ass::asy_deg} holds for $h_l^{(s)}$. 
%\label{ass:check_s}
%\end{lemma}
%Lemma \ref{ass:check_s} allows us to apply Theorem \ref{thm::mainr_deg} to fully degenerate kernel $h_l^(s)$ that has order $s$, and concludes the following asymptotic theory for general $U_{n,q}$. 

\begin{theorem}[General case]
Suppose $h_l$ is a kernel function with order $r$, and has degeneracy of order $s\le r$. In addition, Assumption \ref{ass::lead_hoeff} and \ref{ass::asy_deg} hold for $h_l^{(s)}$. Then we have under the null, for any $q\in2\Z_+$,
$$
[(qs)!]^{-1/2}\binom{r}{s}^{-q}n^{qs/2}\tsig_s^{-1/2}(q) U_{n,q}\cod N(0,1).
$$
Furthermore, for any finite set $I\subseteq 2\Z_+$, $(U_{n,q})_{q\in I}$ are asymptotically jointly independent.
\label{thm::mainr}
\end{theorem}
Recall that we call any kernels with order $r=1$ fully degenerate by our convention. Hence Proposition \ref{thm::mainr_deg} and Theorem \ref{thm::mainr} both contain Theorem \ref{thm::mainr1} as a special case, which includes the mean and covariance testing in \cite{he2021asymptotically} and the spatial sign based test studied in \cite{wang2015high}. Moreover, in Appendix \ref{sec:verify} we demonstrate that the assumptions are satisfied under more broad settings such as simultaneous testing of linear model coefficients and component-wise independence testing based on Kendall's $\tau$ and Spearman's $\rho$.

To end this section, we consider a special case for illustration where the kernel $h_l$ is non-degenerate ($s=1$). As we discussed, it suffices to verify Assumption \ref{ass::asy_deg} (a) and (b1) for $h_l^{(1)}$, which is implied by Assumption \ref{ass::cumr} for $h_l^{(1)}$ (as Assumption 2.1 in \cite{zhang2021adaptive} implies Assumption 6.1 therein). We summarize the result in the following corollary.
\begin{corol}[$s=1$]
Suppose $h_l$ is a non-degenerate kernel with order $r$, and Assumption \ref{ass::lead_hoeff} holds so that $\Sigma_t(q)=o(n^{q(t-1)}\Sigma_1(q))$ for any $1<t\le r$. Moreover, Assumption \ref{ass::cumr} holds for $Y_i=(h_l^{(1)}(X_i))_{l\in\cL}$. Then we have under the null, for any $q\in2\Z_+$,
$$
(q!)^{-1/2}r^{-q}n^{q/2}\Sigma_1^{-1/2}(q) U_{n,q}\cod N(0,1).
$$
Furthermore, for any finite set $I\subseteq 2\Z_+$, $(U_{n,q})_{q\in I}$ are asymptotically jointly independent.
\end{corol}

\subsection{Power Analysis}
\label{sec::power}
In this section, we derive the asymptotic theory of the U-statistic $U_{n,q}$ under the alternative. %In addition to analyzing power, the result is also useful in constructing and analyzing the variance estimator of $U_{n,q}$, which we will see in the next section.
Suppose $\E[h_{l}(X_1,\ldots,X_r)]=\theta_l$. Define the centered U-statistic $\thl=h_l-\theta_l$. We have
\begin{equation}
\begin{aligned}
U_{n,q}=&(P_{qr}^n)^{-1}\sum_{l\in\cL}\sum_{I\in P_{qr}(n)} [(\otimes)^q(\thl+\theta_l)](X_I)\\
=&\|\Theta\|_q^q+\tilde U_{n,q}+\sum_{c=1}^q\binom{q}{c}(P_{cr}^n)^{-1}\sum_{l\in\cL}\sum_{I\in P_{cr}(n)} \theta_l^{q-c}[(\otimes)^c\thl](X_I),
\end{aligned}   
\label{eq::power_decomp}
\end{equation}
where we have viewed the constant $\theta_l$ as a function so that $\theta_l(X_I)\equiv \theta_l$ and $\tilde U_{n,q}$ is the proposed $L_q$-type U-statistic associated with $\tilde h_l$.

Since $\tilde h_l$ has zero mean, we have the asymptotic normality of $n^{qs/2}\tsig_s^{-1/2}(q)\tilde U_{n,q}$, and hence it is natural to define the signal level as $$\gamma_{n,q}=n^{qs/2}\tsig_s^{-1/2}(q)\|\Theta\|_q^q.$$ 
We also make he following assumption to bound the cross-product terms under the local alternative where $\gamma_{n,q}$ goes to some non-zero constant.
\begin{assumption}
\label{ass::power_cross}
Suppose $h_l$ has degeneracy of order $s$. We have for any $q\in 2\Z_+$ and $c=1,\ldots,q-1$,
$$
\sum_{l_1,l_2\in\cL}\theta_{l_1}^{q-c}\theta_{l_2}^{q-c}\sigma_s^c(l_1,l_2)=o\big(n^{-(q-c)s}\tsig_s(q)\big).
$$
\end{assumption}
With Assumption \ref{ass::power_cross}, we have the following theorem under the alternative.
\begin{theorem}
\label{thm::power}
Under the same assumption as Theorem \ref{thm::mainr}. We have the followings for any $q\in2\Z_+$.
\begin{itemize}
    \item Suppose $\gamma_{n,q}\rightarrow \infty$. Then $n^{qs/2}\tsig_s^{-1/2}(q)U_{n,q}\cop\infty$.
    \item Suppose $\gamma_{n,q}\rightarrow 0$. Then $U_{n,q}$ has the same asymptotic distribution as the limiting null.
 \item Suppose Assumption \ref{ass::power_cross} holds. Under the local alternative where $\gamma_{n,q}\rightarrow\gamma\in(0,\infty)$, we have
\begin{equation}
 n^{qs/2}\tsig_s^{-1/2}(q) U_{n,q}\cod N\left(\gamma,(qs)!\binom{r}{s}^{2q}\right).
 \label{eq::asym_alt}
\end{equation}
\end{itemize}
\end{theorem}
As a result of Theorem \ref{thm::power}, when $\gamma_{n,q}\rightarrow\infty$, we have $U_{n,q}/\|\Theta\|_q^q=1+o_p(1)$, and hence $U_{n,q}$ is a ratio-consistent estimator of $\|\Theta\|_q^q$.

Since Lemma 6.1 of \cite{zhang2021adaptive} states that for $s=1$, Assumption \ref{ass::power_cross} is implied by Assumption \ref{ass::cumr} (imposed on $h_l^{(1)}$), we have the following corollary.
\begin{corol}[$s=1$]
\label{corol::powers1}
Suppose $h_l$ is non-degenerate, and Assumption \ref{ass::cumr} holds for $h_l^{(1)}$. Moreover, Assumption \ref{ass::lead_hoeff} also holds. Then we have the first two cases in Theorem \ref{thm::power}. In addition, under the local alternative where $\gamma_{n,q}\rightarrow\gamma\in(0,\infty)$, we have for any $q\in2\Z_+$,
$$
 n^{q/2}\tsig_1^{-1/2}(q) U_{n,q}\cod N\left(\gamma,q!r^{2q}\right).
$$
\end{corol}

To illustrate the power behavior with different $q\in2\Z_+$,  we consider $\Theta=\delta\cdot(\bm{1}_d,\bm{0}_{L-d})^T$ ($L=|\cL|$) for simplicity. In view of (\ref{eq::asym_alt}), to compare the finite-sample power, it suffices to investigate
$$
\frac{\E[U_{n,q}]}{\sqrt{\var(U_{n,q})}}\asymp\frac{n^{qs/2}\tsig_s^{-1/2}(q)\|\Theta\|_q^q} {[(qs)!]^{1/2}\binom{r}{s}^q},
$$
for different $q$. We assume that $\tsig_s(q)\asymp a^q N$ for some $a>0$ and $N=N(L)$ depending on $L$, so that we have
\begin{equation}
\frac{\E[U_{n,q}]}{\sqrt{\var(U_{n,q})}}\asymp [(qs)!]^{-1/2}\binom{r}{s}^{-q}n^{qs/2}a^{-q/2}N^{-1/2}\delta^qd.
\label{eq::evratio}
\end{equation}
To achieve (asymptotic) power $\Phi(R-z_{1-\alpha})$ for some $R>0$, we let $\E[U_{n,q}]/\sqrt{\var(U_{n,q})}=R$ and solve $\delta=\delta(q)$ from (\ref{eq::evratio}) to obtain
\begin{equation}
\delta(q)\asymp \binom{r}{s}\sqrt{a}[(qs)!]^{1/2q}(\sqrt{N}R/d)^{1/q}n^{s/2}.
\label{eq::deltaq}
\end{equation}
Given $\delta=\delta(q)$, the U-statistic $U_{n,q}$ has fixed asymptotic power $\Phi(R-z_{1-\alpha})$. Therefore, we hope to find $q\in 2\Z_+$ associated with the smallest $\delta(q)$, which tends to have highest power. Note that we consider one-sided test, since even $q$ implies that the $L_q$ norm of the parameter is always nonnegative. %Here we have assumed for simplicity that we use one-sided test for all $q$'s. However, in practice we should use two-sided test for odd $q$ and one-sided  for even $q$ to obtain highest power. 
We have the following result that mimics Proposition 2.3 in \cite{he2021asymptotically}.
\begin{proposition}
The optimal $q$ that minimizes $\delta(q)$ in (\ref{eq::deltaq}) is achieved at
\begin{itemize}
    \item 2, if $d\ge \sqrt{N}R$;
    \item some $q\in 2\Z_+$, and is increasing in $\sqrt{N}R/d$ if $d>\sqrt{N}R$.
\end{itemize}
\label{prop::optim_q}
\end{proposition}
\begin{remark}
 Proposition \ref{prop::optim_q} is derived under the assumption $\Theta=\delta(\bm{1}_d,\bm{0}_{L-d})$, and  can be straightforwardly extended to more general $\Theta$, e.g. $\theta_j=\delta r^j$ (or $\delta j^a$) for some constant $r$ (or $a$) and $j\le d$, as long as $\|\Theta\|_q$ can be calculated explicitly. It basically demonstrates that when $\Theta$ is dense, the optimal $q$ is 2. As $\Theta$ gets sparser, the optimal $q$  gets larger. This coincides with our intuition.   
 
% The analysis gets slightly more complicated when the signs of the components vary. In the extreme case, the positive components of $\Theta$ can nearly cancel out with the negative ones, and yield $\|\Theta\|_1\approx 0$ (may also hold for $L_q$ norm with odd $q$). For example, if we consider $\Theta=\delta(\bm{1}_{d/2},-\bm{1}_{d/2},\bm{0}_{L-d})$, Proposition \ref{prop::optim_q} is replaced by $q\in2\N$ (and optimal $q$ for $d\ge \sqrt{N}R$ is 2), as $\|\Theta\|_q=0$ for any odd $q$ and the associated statistic has asymptotically trivial power in this case.
\end{remark}
%In practice, we recommend a combination of small and large even $q$'s for best empirical performance (say $I=(2,4)$ or $(2,6)$). %To gain highest power against sparse alternative with homogeneous directions, we may also consider adding $q=1$, e.g. $I=(1,2,6)$.

\section{Test Statistics and Computation}
\label{sec::dpustat}
\subsection{Computation for Order-One Kernels with Dynamic Programming}
When $h_l$ has order 1, i.e. $r= 1$, we have
$$
U_{n,q}=\sum_{l\in\cL}(P^n_{q})^{-1}\sum^*_{1\le i_1,\ldots,i_q\le n}\prod_{c=1}^q h_l(X_{i_c}).
$$
In this case, we can apply dynamic programming to speed up the computation and calculate the statistic exactly. In particular, we define for $q\le m \le n$
$$
D_{q,l}(m)=\sum_{1\le i_1<\cdots<i_q\le m}\prod_{c=1}^q h_l(X_{i_c}).
$$
We have $U_{n,q}=\binom{n}{q}^{-1}\sum_{l\in\cL} D_{q,l}(n).$ Fix a component $l\in \cL$. Note that $D_{1,l}(m)=\sum_{i=1}^m h_l(X_{i})$ is simply the cumulative sum of the sequence $\{h_l(X_i)\}_{i=1}^n$, and can be calculated efficiently in $O(n)$. Afterwards, we can use the following approach to calculate $D_{c,l}$ from $D_{c-1,l}$ recursively,
$$
D_{c,l}(m)=D_{c,l}(m-1)+D_{c-1,l}(m-1)h_l(X_m),\quad m\ge c,
$$
with $D_{c,l}(m)=0$ for $1\le m<c$. We run this step for $c=2,\ldots,q$ and each step takes $O(n)$ time.

In summary, the total cost of calculating $D_{q,l}(n)$ and $U_{n,q}$ are $O(qn)$ and $O(qn|\cL|)$, respectively, which is a huge reduction from brute-force calculation that requires $O(qn^{q}|\cL|)$ computation.

\subsection{An Asymmetric U-statistic for Kernels with Higher Order}
The method in last section only works for $h_l$ with order 1. However, it motivates us to consider the following U-statistic for $r>1$, to which dynamic programming can be applied.
$$
U_{n,q}^M=\sum_{l\in\cL}\binom{n}{qr}^{-1}\sum_{1\le i_1<\cdots<i_{qr}\le n}\prod_{c=1}^q h_l(X_{i_{(c-1)r+1}},\ldots,X_{i_{cr}}).
$$
Here we only consider all monotone indices in each summand, so that dynamic programming is applicable as we discuss next. We use the superscript $M$ to emphasize the monotonicity, and to distinguish it from $U_{n,q}$.
Note that $U_{n,q}^M$ is essentially a U-statistic with asymmetric kernel. 

To see the difference between $U_{n,q}$ and $U_{n,q}^M$, consider the case $r= 2$ and $q=2$. Then $U_{n,q}$ is essentially U-statistic of order $q\cdot r=4$. In this case, the (symmetric) kernel for $U_{n,q}$ is given by
$$\frac{1}{3}\big[h_l(X_1,X_2)h_l(X_3,X_4)+h_l(X_1,X_3)h_l(X_2,X_4)+h_l(X_1,X_4)h_l(X_2,X_3)\big],$$ 
while that for $U_{n,q}^M$ equals $h_l(X_1,X_2)h_l(X_3,X_4)$. Although $U_{n,q}^M$ is also a U-statistic of order 4, the number of unique summand in $U_{n,q}$ is 3 times that of $U_{n,q}^M$.

To apply dynamic programming, we define for $qr\le m\le n$,
$$D_{q,l}^M(m)=\sum_{1\le i_1<\cdots<i_{qr}\le m}\prod_{c=1}^q h_l(X_{i_{(c-1)r+1}},\ldots, X_{i_{cr}}).$$
We first initialize $D_{0,l}^M(m)\equiv 1$ for $1\le m\le n$. Then we use the following recursive updating formula for $c=1,\ldots,q$ and $m\ge cr$,
$$
D_{c,l}^M(m)=D_{c,l}^M(m-1)+\sum_{(c-1)r<i_1<\cdots<i_{r-1}<m}D^M_{c-1,l}(i_1-1)h_l(X_{i_1},\cdots,X_{i_{r-1}},X_m),
$$
with $D^M
_{c,l}(m)=0$ for $1\le m<cr$. Note that the calculation for each $c,m$ is $O(m^{r-1})$. Therefore, the total costs of obtaining $D_{q,l}^M(n)$ and $U_{n,q}^M$ are $O(qn^r)$ and $O(qn^r|\cL|)$ respectively, much better than the brute force approach with $O(qn^{qr}|\cL|)$ computation.

As an illustrating example that summarizes the idea of the recursive formula, again consider the case $r=2,q=2$. Suppose we have calculated $D_{2,l}^M(m)$ for $m<n$ and the full sequence $D_{1,l}^M$. To obtain $D_{2,l}^M(n)$, we only need to calculate the difference between $D_{2,l}^M(n)$ and $D_{2,l}^M(n-1)$, which consists of terms of form $h_l(X_i,X_j)h_l(X_k,X_n)$ for $1\le i<j<k<n$. If we fix $k$, we can factor $h_l(X_k,X_n)$ out, the summation over $h_l(X_i,X_j)$ is simply $D_{1,l}^M(k-1)$ and we obtain
$$
D_{2,l}^M(n)=D_{2,l}^M(n-1)+\sum_{k=3}^{n-1}D^M_{1,l}(k-1)h_l(X_k,X_n).
$$
It is not difficult to see that the asymptotic variance of $U_{n,q}^M$ has order $n^{qs/2}\Sigma_r(q)$, which is proportional to $\Sigma_s$ instead of $\tsig_s$ as the latter describes the variance of the symmetric kernel $H_{l,q}^{(s)}$ and hence $U_{n,q}$. However, as the leading term in the Hoeffding-type decomposition of $U_{n,q}^M$ is a weighted U-statistic, it is technically difficult to derive the exact form of the asymptotic variance, but we still expect the asymptotic normality and independence to hold, and we conjecture that there exists some constant $C_{q,s}^M$ such that
$$
n^{qs/2}\Sigma_s^{-1/2}(q) U_{n,q}^{M,(qs)}\cod N(0,C_{q,s}^M),
$$
and we leave the rigourous proof for future research. If $h_l$ is fully degenerate, i.e. $s=r$, then it is not difficult to see that $C_{q,s}^M=(qr)!$ as $U_{n,q}$ and $U_{n,q}^M$ coincides in this case. In Section \ref{sec::var_est}, we propose a permutation based variance estimation procedure, which avoids calculating the exact variance.

\subsection{Variance Estimator}
\label{sec::var_est}
In Section \ref{sec::cltr}, we have shown that
$$
[(qs)!]^{-1/2}\binom{r}{s}^{-q}n^{qs/2}\tsig_s^{-1/2}(q) U_{n,q}\cod N(0,1).
$$
However, as $\tsig_s(q)$ is unknown, we shall propose an unbiased estimator $\hsig_s(q)$ for $\tsig_s(q)$, with which we may define the following studentized statistic,
\begin{equation}
T_{n,q}=[(qs)!]^{-1/2}\binom{r}{s}^{-q}n^{qs/2}\hsig_s^{-1/2}(q) U_{n,q}
\label{eq::tnq}
\end{equation}
Interestingly, the estimator $\hsig_s(q)$ still falls into the framework introduced in this paper, i.e., it is the proposed U-statistic $U_{n,q}$ associated with some kernels derived from $h_l$. Therefore, the ratio-consistency is guaranteed by the theories in Section \ref{sec::cltr}.

For illustration, consider $h_l$ with order 1. In this case, we have
$$
\tsig_1(q)=\sum_{l_1,l_2\in\cL}\cov^q\big(h_{l_1}(X_1),h_{l_2}(X_1)\big).
$$
Consider $\cL^2=\{\bl=(l_1,l_2):l_1,l_2\in\cL\}$, and  $\Theta=\big(\theta_{\bl}\big)_{\bl\in\cL^2}$ with $\theta_{\bl}=\cov\big(h_{l_1}(X_1),h_{l_2}(X_1)\big)$, and the kernel $g_{\bl}(X_1)=h_{l_1}(X_1)h_{l_2}(X_1)$.

We can estimate $\tsig_1(q)=\|\tilde\Theta\|_q^q$ with the proposed U-statistic associated with $g_{\bl}$, i.e.
$$
 \sum_{\bl\in\cL^2}(P^n_q)^{-1}\prod_{c=1}^q g_{\bl}(X_c)=\sum_{l_1,l_2\in\cL}(P^n_q)^{-1}\prod_{c=1}^q h_{l_1}(X_c)h_{l_2}(X_c).
$$
This is an unbiased estimator under the null where $\E[h_l]\equiv 0$. To handle the non-zero mean under the alternative, we simply center $h_l$ and define 
$\tilde h_l(X_1)=h_l(X_1)-\frac{1}{n}\sum_{i=1}^n h_l(X_i).$
We modify the kernels $g_{\bl}$ correspondingly, and define $\tilde g_{\bl}(X_1)=\tilde h_{l_1}(X_1)\tilde h_{l_2}(X_1)$.
Then we may estimate $\tsig_1(q)$ with the associated statistic
$$
 \hsig_1(q)\deltaeq\sum_{\bl\in\cL^2}(P^n_q)^{-1}\prod_{c=1}^q \tilde g_{\bl}(X_c)=\sum_{l_1,l_2\in\cL}(P^n_q)^{-1}\prod_{c=1}^q \tilde h_{l_1}(X_c)\tilde h_{l_2}(X_c).
$$
It is expected that $\hsig_1(q)$ is a ratio-consistent estimator for $\tsig_1(q)$, although a formal justification is beyond the scope of this paper.

In general, for any non-degenerate ($s=1$) kernels with order $r>1$, we may follow a similar procedure. Under the null, we may define the kernel $g_{\bl}(X_1)=h_{l_1}^{(1)}(X_1)h_{l_2}^{(1)}(X_1)$. When the exact form of $h_l^{(1)}$ is unknown, one approach is to estimate it based on the sample and derive a plug-in estimator. An alternative approach is to apply the fact that
$$
\E[h_{l_1}^{(1)}(X_1)h_{l_2}^{(1)}(X_1)]=\E[h_{l_1}(X_1,\ldots,X_r)h_{l_2}(X_r,\ldots,X_{2r-1})],
$$
and define $g_{\bl}=h_{l_1}(X_1,\ldots,X_r)h_{l_2}(X_r,\ldots,X_{2r-1})$ as a kernel with order $2r-1$ instead. To guarantee the consistency under the alternative, we have to define $\tilde g_{\bl}$, the counterpart of $g_{\bl}$ obtained by centering $h_l^{(1)}$ or $h_l$. Moreover, we can also consider the U-statistic with monotone indices, i.e. $U_{n,q}^M$  to save computational cost, since we are only interested in obtaining a ratio-consistent estimator for $\tsig_1(q)$.

When $s>1$, the form becomes even more complicated. Since the computation of $U_{n,q}$ with large $q$ is generally infeasible for $r\ge 2$, we recommend using $U_{n,q}^M$  to construct the test in practice, and hence we do not discuss the variance estimation for $U_{n,q}$ in the case $s\ge 2$.

For $U_{n,q}^M$, as discussed in Section \ref{sec::dpustat}, it is expected to be asymptotically normal with asymptotic variance proportional to $\Sigma_s(q)$, and we follow some non-parametric approaches to estimate the asymptotic variance. One approach is to use permutation based variance estimator. This applies to any testing problems where the distribution of the test statistic is invariant up to certain permutations under the null, and hence the proposed U-statistics computed on the permuted data has the same distribution as the null distribution. For example, any test for component-wise independence satisfies this property. The permutation based method is also applicable to linear model coefficient testing example (under certain model assumptions) introduced in Section \ref{sec::eg}.

For illustration, let $\sigma:[n]\rightarrow [n]$ denote a permutation on $[n]$.
To simulate the distribution of the following U-statistic, which was proposed for testing the linear model coefficients. 
\begin{align*}
U_{n,q}^M=&(P^n_{2q})^{-1}\sum_{l\in\cL}\sum_{1\le i_1<\cdots<i_{2q}\le n}2^{-q}\prod_{c=1}^q [(X_{i_{2c-1}}-X_{i_{2c}})(Y_{i_{2c-1}}-Y_{i_{2c}}-(X_{i_{2c-1}}-X_{i_{2c}})^T\bbeta_0)]_l\\
=&(P^n_{2q})^{-1}\sum_{l\in\cL}\sum_{1\le i_1<\cdots<i_{2q}\le n}2^{-q}\prod_{c=1}^q [(X_{i_{2c-1}}-X_{i_{2c}})(E_{i_{2c-1}}-E_{i_{2c}})]_l,
\end{align*}
where we have defined $E_i\deltaeq Y_i-X_i\bbeta_0$ for $i=1,\ldots,n$. Then we may consider
$$
U_{n,q}^{M,\sigma}=(P^n_{2q})^{-1}\sum_{l\in\cL}\sum_{1\le i_1<\cdots<i_{2q}\le n}\prod_{c=1}^q h_l\big((X_{i_{2c-1}},Y_{\sigma(i_{2c-1})}),(X_{i_{2c}},Y_{\sigma(i_{2c)}})\big),
$$
for any random permutation $\sigma$. Under the null $\bbeta=\bbeta_0$,  $E_i=\varepsilon_i$ is i.i.d. and independent of $X_i$. Therefore, the distribution of $U_{n,q}^{M,\sigma}$ coincides with that of $U_{n,q}^{M}$ under the null, and can be used to estimate the asymptotic variance. To be more precise, consider $B$ random permutations $\sigma_1,\ldots,\sigma_B$. We estimate $\var(U_{n,q}^M)$ by the permutation based sample variance, i.e.
$$
V_{n,q}^M=\frac{1}{B-1}\sum_{b=1}^B\left(U_{n,q}^{M,\sigma_b}-\frac{1}{B}\sum_{b'=1}^BU_{n,q}^{M,\sigma_{b'}}\right)^2.
$$
The studentized statistic is then defined as $T_{n,q}^M=\frac{U_{n,q}^M}{\sqrt{V_{n,q}^M}}.$ We may also conduct permutation based test by obtaining the critical value directly from the empirical distribution of $\{U_{n,q}^{M,\sigma_b}\}_{b=1}^B$. Empirically, we find that two approaches have similar performance when $B=100$, and hence we only report the result of $T_{n,q}^M$ in the simulations. 

We remark that if we have computational power to calculate the full U-statistic $U_{n,q}$, this permutation based approach is still applicable to estimating the variance  of $U_{n,q}$, and we simply replace $U_{n,q}^{M,\sigma}$ by $U_{n,q}^\sigma$, that is, the full U-statistic calculated on permuted sample.

The other possible approach is the (multiplier) bootstrap method. % Then under certain settings, we can simulate the null distribution of $U_{n,q}^M$ by
%$$
%U_{n,q}^{M,e}=\sum_{l\in\cL}(P^n_{qr})^{-1}\sum_{1\le i_1<\cdots <i_{qr}\le n}\prod_{c=1}^q h_l(X_{i_{(c-1)r+1}},\ldots,X_{i_{cr}})e_{i_{(c-1)r+1}}\cdots e_{i_{cr}}.
%$$
However, the theoretical analysis for bootstrap based U-statistic with high orders is challenging \citep{arcones1992bootstrap}, and in this paper we rely mainly on the permutation based approach.

\subsection{Adaptive Testing Procedure}
Let $I$ be a set of $q\in2\Z_+$ (e.g. \{2,6\}). Under the null $T_{n,q}$'s are asymptotically independent for different $q\in I$, we may construct an adaptive test straightforwardly by combining $q \in I$. Denote the $p$-values associated with $T_{n,q}$ as $p_q$, then we may define the statistic of the adaptive test associated with $I\subseteq 2\Z_+$ as $p_{ada}=\min_{q\in I}p_q$, and its corresponding $p$-value equals to $1-(1-p_{ada})^{|I|}$.

Equivalently, if we want to conduct a level-$\alpha$ test, we may define the adaptive test function with $I$ as
$$
\phi_{I,\alpha}=\max_{q\in I}\phi_{q,1-(1-\alpha)^{1/|I|}},
$$
where $\phi_{q,\alpha}$ is the level-$\alpha$ test based on $T_{n,q}$, i.e. $\phi_{q,\alpha}=\mathbbm{1}(T_{n,q}>\Phi(1-\alpha))$ with $\Phi(x)$ denoting the cdf of standard normal.
Therefore, the adaptive test rejects the null if one of $T_{n,q}$'s exceeds its critical value $\Phi((1-\alpha)^{1/|I|})$.
Similar procedure can be straightforwardly applied to $T_{n,q}^M$ as well. %Note that we have assumed one-sided test for all $q$'s as we did in Section \ref{sec::power} to ease the notation, although two-sided test is recommended for odd $q$ in practice.

In general, a smaller $q$ (say $q=2$) has higher power against dense alternative, while a larger $q$ is favored under sparse alternative. With adaptive test, we can combine the advantages of different $q$'s and achieve high power against both dense and sparse alternatives, and the power of the adaptive test goes to 1 as long as one of the test statistics $T_{n,q}$ has asymptotic power 1. Since the corrected level $1-(1-\alpha)^{1/|I|}$ is decreasing in $I$, it may hurt the overall power of the adaptive test to include too many $q$ in $I$, even though each single-$q$ based test might be optimal for some particular alternative. To have better finite sample performance, we recommend to combine two single-$q$ based tests in practice, which has relatively high power against alternatives with different sparsity.

\subsection{Two Sample Asymmetric U-statistics}
In this section, we introduce the two-sample U-statistics with monotone index, which the dynamic programming can also be applied to. We mainly consider kernels with order (1,1) for illustration, as the main motivating example is the two-sample spatial sign test \citep{chakraborty2017tests}. Suppose we observe two i.i.d. samples $X_1,\ldots,X_n\in\R^p$ and $Y_1,\ldots,Y_m\in\R^p$. We want to test some high dimensional parameter $\Theta=\Theta(X_i,Y_i)=\{\theta_l:l\in\cL\}=\bz$. Recall the full U-statistic is defined as
$$U_{n,m,q}=\sum_{l\in\cL}(P^n_{q}P^m_{q})^{-1}\sum^*_{1\le i_1,\ldots,i_{q}\le n}\sum^*_{1\le j_1,\ldots,j_{q}\le m}\prod_{c=1}^q h_l(X_{i_c},Y_{j_c}).$$
In particular, for the spatial sign test, we take $\cL=\{1,\ldots,p\}$, and $h_l(X_i,Y_j)=(x_{i,l}-y_{j,l})\,/\,\|X_i-Y_j\|$.

As the computation of $U_{n,m,q}$ is $O(n^qm^q)$, we consider the monotonically-indexed counterpart as follows,
$$U_{n,m,q}^M=\sum_{l\in\cL}\binom{n}{q}^{-1}\binom{m}{q}^{-1}
\sum_{1\le i_1<\cdots<i_{q}\le n}\sum_{1\le j_1<\cdots<j_{q}\le m}\prod_{c=1}^q h_l(X_{i_c},Y_{j_c}).$$
With dynamic programming, we can speed up the computation of $U_{n,m,q}^M$ to $O(mn)$. For $q\le i\le n$ and $q\le j\le m$, define
$$S_{q,l}^M(i,j)=\sum_{1\le i_1<\cdots<i_{q-1}< i}\sum_{1\le j_1<\cdots<j_{q-1}< j}h_l(X_i,Y_j)\prod_{c=1}^{q-1} h_l(X_{i_c},Y_{j_c}),$$
which involves all the summands in $U_{n,m,q}^M$ s.t. the largest indices of $X$ and $Y$ are $i$ and $j$, respectively. We also define
$$D_{q,l}^M(i,j)=\sum_{1\le i_1<\cdots<i_q\le i}\sum_{1\le j_1<\cdots<j_q\le j}\prod_{c=1}^{q} h_l(X_{i_c},Y_{j_c}).$$
If we view $S_{q,l}(i,j)$ as a matrix indexed by $i\in[1,n],j\in[1,m]$, then $D_{q,l}(i,j)$ adds up all the elements that is top and left to the $(i,j)$-th element.
In particular, we have $$U_{n,m,q}^M=\binom{n}{q}^{-1}\binom{m}{q}^{-1}\sum_{l\in\cL}D_{q,l}^M(n,m),$$
and $S_{q,l}^M(i,j)=D_{q-1,l}^M(i-1,j-1)h_l(X_i,Y_j)$.
We first initialize $D_{0,l}^M(i,j)\equiv 1$ for all $i,j$. Then we use the following recursive updating formula for $c=1,\ldots,q$ and $i,j\ge c$,
\begin{align*}
 D_{c,l}^M(i,j)=&S_{c,l}^M(i,j)+D_{c,l}^M(i-1,j)+D_{c,l}^M(i,j-1)-D_{c,l}^M(i-1,j-1)\\
 =&D_{c-1,l}^M(i-1,j-1)h_l(X_i,Y_j)+D_{c,l}^M(i-1,j)+D_{c,l}^M(i,j-1)-D_{c,l}^M(i-1,j-1).
\end{align*}

We conjecture that the statistics $U_{n,m,q}$ and $U_{n,m,q}^M$ are both asymptotically normal. Their variances can be estimated by permutation (i.e. randomly permuting the data between two samples), similar to the one-sample setting.

\section{Simulation Studies}
\label{sec::sim}
In this section, we examine the finite sample performance of the proposed  single-$q$ and adaptive tests via simulations. Throughout this section,  for kernels with $r=1$ such as the spatial sign test, we use $T_{n,q}$ based on the variance estimator introduced in Section \ref{sec::var_est}; for $r>1$, we construct the test with $T_{n,q}^M$ using the permutation based variance estimator. Section \ref{sec::sim_spatial} presents the simulation results for one-sample test of spatial sign, and Section \ref{sec::sim_ts_spatial} includes its two-sample counterpart. Section \ref{sec::sim_ind} contains the results for testing component-wise independence, and Section \ref{sec::sim_lr} corresponds to testing the nullity of linear model coefficients. For all testing problems, we perform 1000 Monte Carlo replications. 

\subsection{Tests for Spatial Sign}
\label{sec::sim_spatial}
In this subsection, we want to test
$$
\cH_0: \E[X_i/\|X_i\|\,]=\bm{0}\quad v.s. \quad \cH_a: \E[X_i/\|X_i\,\|]\not=\bm{0}.
$$
This is similar to the standard mean-testing problem $
\cH_0: \E[X_i]=\bm{0}$ and $\cH_a: \E[X_i]\not=\bm{0}$,
with certain elliptical symmetry assumption, but does not require the existence of first moment. \cite{wang2015high} extended the U-statistic based two sample test by \cite{chenqin2010} and proposed the $L_2$-norm based spatial sign test statistic, which is indeed a special case of our U-statistic with kernel $h_l(X_i)=x_{i,l}/\|X_i\|$ and $q=2$, although they use a different variance estimator. Through both theory and simulations, \cite{wang2015high} found that their test has slight power loss compared to the test by \cite{chenqin2010} for Gaussian data, but has much higher power for heavy-tailed data.  However, the  $L_2$-norm based spatial sign test mainly targets at dense alternatives, and their performance under the sparse alternative is not yet examined.  
%they do not investigate the power against sparse alternatives, which can be of interest.
%includes the simulation studies for the proposed statistics applied to spatial sign test \citep{wang2015high}, which can be viewed as a more general test to the mean testing problem. To be more precise, we want to test 

Below we generate the simulated data from the model $X_i= Z_i+\mu$, where  $Z_i\in \R^p$ has i.i.d. standard normal or $t_3$ components. %We take $\mu=\delta(\bm{1}_r^T,\bm{0}_{p-r}^T)^T$ and the variance matrix $\Sigma=\Gamma\Gamma^T=I_p$. 
We include both sparse and dense alternatives, (corresponding to $r=2$ and $r=p$, respectively) and also consider both light and heavy-tailed distributions (i.e., standard normal versus $t_3$). Table~\ref{tb:spatialsign} collects the size and power results. Among all the tests being compared, we mention that mean test with $q=2$ corresponds to the test by \cite{chenqin2010} (up to a different variance estimator), spatial sign test with $q=2$ corresponds to the test by \cite{wang2015high} (up to a different variance estimator). Our test with $q = 2$ performs very similar to the test by \cite{chenqin2010} (for mean test) and the test by \cite{wang2015high} (for spatial sign test), and therefore we omit the results of their tests.

%to see if the spatial sign test with large $q$ has similar properties under sparse alternatives.

\begin{table}[H]
\centering
\begin{tabular}{|c|c|c|c|c|c|c|c|}
\hline
DGP & $\delta,r$ & Test & $q=2$ & $q=4$ & $q=6$ & $q=2,4$ & $q=2,6$\\ \hline
\multirow{6}{*}{Gaussian} & \multirow{2}{*}{0, NA} &  mean & 6.1 &  5.9 & 3.0 & 6.4 & 5.4 \\ \cline{3-8}
& &  spatial sign & 5.8 & 6.1 
& 3.2 & 6.4 & 5.6\\ \cline{2-8}
& \multirow{2}{*}{0.3, 2} &  mean & 74.3 & 96.3 & 90.0 & 96.3 & 91.4 \\ \cline{3-8}
& & spatial sign & 74.5 & 96.0 & 90.0 & 96.2 & 92.1\\ \cline{2-8}
& \multirow{2}{*}{0.1, $p$} &  mean & 100 & 92.2 & 37.8 & 100 & 100\\ \cline{3-8}
&  & spatial sign & 100 & 92.1 & 37.3 & 100 & 100\\ \hline
\multirow{6}{*}{$t_3$} & \multirow{2}{*}{0, NA} &  mean & 6.1 & 4.5 & 1.9 & 6.8 & 4.8\\ \cline{3-8}
& &  spatial sign & 6.2 & 4.8 & 2.1 & 6.9 & 4.9\\ \cline{2-8}
& \multirow{2}{*}{0.4, 2} &  mean & 42.7 & 73.1 & 57.6 & 72.7 & 63.6\\ \cline{3-8}
& & spatial sign & 51.1 & 79.3 & 65.4 & 80.1 & 71.1 \\ \cline{2-8}
& \multirow{2}{*}{0.1, $p$} & mean & 98.2 & 30.1 & 14.5 & 97.0 & 97.0\\ \cline{3-8}
&  & spatial sign & 99.8 & 36.5 & 15.6 & 99.3 & 99.3\\ \hline
\end{tabular}
\caption{Size and power in percentage for mean and spatial sign test}
\label{tb:spatialsign}
\end{table}

As seen from Table~\ref{tb:spatialsign}, the size for smaller $q$ (i.e., $q=2,4$) tends to be more accurate than that for large $q$ (i.e., $q=6$), which exhibits some undersize. The adaptive test based on $q=2,4$ has slight oversize whereas the size for $q=2,6$ appears accurate.   In terms of the power, 
%We can see that there are some size distortion for larger $q$, say $q=6$, for both tests. 
 single $q$ tests based on large and small $q$'s correspond to high power against sparse and dense alternatives, respectively, as expected. 
 Moreover, for both alternatives and all $q$, there is no obvious power loss when we compare the spatial sign test  to the mean test for Gaussian data, but the former has notable higher power for $t_3$-distributed data. This suggests that the power robustness and advantage of spatial-sign test based on $L_2$ norm, as discovered first in \cite{wang2015high}, carries over to the more general $L_q$-norm setting. Finally, we see that the adaptive tests have overall good power against both sparse and dense alternatives, and they are always very close to the best single $q$ test in power. It appears that  the adaptive test with $q=$(2,4) slightly outperforms its counterpart with $q=$(2,6) in the sparse regime, which might be related to its slight oversize for both settings.

\subsection{Tests for Component-wise Independence}
\label{sec::sim_ind}

In this subsection, we compare several tests for the componentwise independence based on  Kendall's $\tau$ statistic.
We generate i.i.d. data $X_1,\ldots,X_n\in\R^p$ according to the following two settings used in \cite{leung2018testing}.  
\begin{itemize}
    \item Gaussian vectors with banded covariance matrix: $X_i\iid N(0,\Sigma)$ for $i=1,\ldots,n$,
    where $\Sigma=(\sigma_{ab})_{a,b=1}^p$ with $\sigma_{ab}=(1-\delta)1\{a=b\}+\delta1\{a\le r,~b\le r\}$.
    \item $t$-distribution with banded covariance matrix: $X_i\iid \Sigma^{1/2}Z$ for $i=1,\ldots,n$, where $\Sigma$ is given above, and $Z_{ij}\iid t_3;~i=1,\ldots,n;j=1,\ldots,p$.
\end{itemize}
We set $\delta=0$ under the null, $(\delta,r)=(0.5,3),\,(0.3,4),\,(0.25,6),\,(0.15,10),\,(p,0.05)$ for alternatives with different sparsity levels, and $(n,p)=(100,50)$.
%In particular, we take $\delta=0$ under the null and $\delta>0$ under the alternative.

Among the tests we compare, we include our test $U_{n,q}^M$ with a permutation-based variance estimator, the Kendall's $\tau$ based maximum-type statistic introduced in \cite{han2017distribution} (denoted as $L_{\infty}$ in Table~\ref{tb:kendalltau}), the Kendall's $\tau$ based sum-of-squares-type statistic proposed in \cite{leung2018testing} (denoted as $L_2$ in Table~\ref{tb:kendalltau}), a simple combination of $L_2$ and $L_{\infty}$ based on Bonferroni correction (denoted as $L_2\& L_{\infty}$ in Table~\ref{tb:kendalltau}). 

Note that the $L_2$ test by \cite{leung2018testing} includes the diagonal terms, and is only asymptotically unbiased.  They also have an unbiased $L_2$-type statistic, which is a special case of our  proposed full U-statistic with $q=2$, and its performance is very similar to the biased version according to the simulations, so is not reported. \cite{leung2018testing}  also introduce an efficient algorithm based on inclusion-exclusion law to reduce the computational complexity from $O(n^{2r})$ to $O(n^r)$ ($r=2$ in this example). However, this idea can only be applied to $q=2$, and in general, their idea can speed up the computation from $O(n^{qr})$ to $O(n^{(q-1)r})$ for the exact U-statistic. Our monotonically-indexed U-statistic $U_{n,q}^M$ loses a constant ratio of efficiency compared to $U_{n,q}$, but can reduce the computational burden to $O(n^r)$ for any $q$.

%Note that \cite{leung2018testing} also propose an unbiased U-statistic based on $L_2$ norm, which is identical to our exact statistic with $q=2$, i.e., $U_{n,2}$. 

%Since they have similar performance, we only study the biased version in the simulations.

%, whose kernel functions are of order 2. We may approximate the asymptotic variance based on permutation using the idea in Section \ref{sec::var_est}. Under the null, we have the independent between components of each $X_i$. Therefore, if we write the data as matrix form $X\in \R^{n\times p}$, we can shuffle within each column simultaneously, and the U-statistics for the permuted data have the same distributions, which can be used to simulate the null distribution, or estimate asymptotic variance under the null.
 
% We compare our method with . We also compare our method with \cite{leung2018testing}, which also use an $L_2$-type statistic that includes the diagonal terms, and is only asymptotically unbiased. Note that they also have an unbiased $L_2$-type statistic, which is a special case of the proposed full U-statistic with $q=2$. Since they have similar performance, we only study the biased version in the simulations. 

\begin{table}[H]
\label{tb:kendalltau}
\centering
\begin{tabular}{|c|c|c|c|c|c|c|c|c|c|}
\hline
DGP & $\delta,r$ & $q=2$ & $q=4$ & $q=6$ & $q=2,4$ & $q=2,6$ & $L_2$ & $L_\infty$ & $L_2\&L_\infty$\\ \hline
\multirow{7}{*}{Gaussian} & 0, NA & 5.5 & 5.0 & 5.8 & 6.4 & 7.3 & 5.8 & 4.6 & 5.6 \\ \cline{2-10}
%& 0.5,2 & 7.2 & 14.2 & 66.6 & 70.6 & 60.8 & 70.2 & 14.4 & 83.6 & 77.4\\ \cline{2-10}
& 0.5, 3 & 38.8 & 96.2 & 95.6 & 96.4 & 95.2 & 43.8 & 98.6 & 97.2 \\ \cline{2-10}
& 0.3, 4 & 25.4 & 42.6 & 34.0 & 45.0 & 38.8 & 29.0 & 39.4 & 41.4 \\ \cline{2-10}
& 0.25, 6 & 43.8 & 47.0 & 30.0 & 56.4 & 48.0 & 51.8 & 37.0 & 53.2 \\ \cline{2-10}
& 0.15, 10 & 48.4 & 22.2 & 11.6 & 48.2 & 45.4 & 53.6 & 14.6 & 50.0 \\ \cline{2-10}
& 0.05, $p$ & 95.4 & 14.4 & 9.0 & 93.6 & 93.8 & 96.2 & 13.0 & 94.6\\ \hline
\multirow{7}{*}{$t_3$} & 0, NA & 4.3 & 4.9 & 4.7 & 5.2 & 5.1 & 5.2 & 2.9 & 4.3 \\ \cline{2-10}
%& 0.5,2 & 7.8 & 16.2 & 81.2 & 86.6 & 76.8 & 85.0 & 16.8 & 95.8 & 94.2\\ \cline{2-10}
& 0.5, 3  & 54.0 & 99.8 & 99.6 & 99.8 & 99.8 & 60.0 & 100 & 99.8 \\ \cline{2-10}
& 0.3, 4 & 40.4 & 72.0 & 63.2 & 70.2 & 64.2 & 44.2 & 71.2 & 68.4 \\ \cline{2-10}
& 0.25, 6 & 71.2 & 80.0 & 64.4 & 83.6 & 77.4 & 76.2 & 66.8 & 78.4 \\ \cline{2-10}
& 0.15, 10 & 78.6 & 47.6 & 22.6 & 76.0 & 73.0 & 83.2 & 27.8 & 76.8 \\ \cline{2-10}
& 0.05, $p$ & 100 & 30.8 & 16.0 & 100 & 100 & 100 & 23.8 & 100 \\ \hline
\end{tabular}
\caption{Size and power in percentage for Kendall's $\tau$ tests}
\end{table}

 In Table~\ref{tb:kendalltau}, we report the size and power of all Kendall $\tau$ based tests for a range of alternatives with varying degrees of sparsity.
 %based on 1000 Monte Carlo simulations, and power based on 500 replications. (TBD)
As we can see, the sizes for most tests are accurate, except for the adaptive test with $q=(2,6)$ for Gaussian data and the $L_{\infty}$ method for the $t_3$ case. 
In terms of the power, the proposed statistic with $q=2$ and $L_2$ method in \cite{leung2018testing} are favored for dense alternative, whereas our statistics based on large $q$ or $L_\infty$ method in \cite{han2017distribution} have higher power under sparse alternative. The $L_2$ method slightly outperforms our studentized $U_{n,q}^M$ (with $q=2$), which is expected since the monotonically indexed-U statistic is less (statistically) efficient compared to full $U$ statistic.% but gains computational efficiency.
%presumably due to the slight loss of efficiency when we 
The adaptive test based on $q=(2,4)$ seems to have comparable performance with $L_2\&L_{\infty}$ (i.e., a simple combination of $L_2$ and $L_\infty$ based on naive Bonferroni correction), although the asymptotic independence between $L_2$ and $L_{\infty}$ statistics based on Kendall's $\tau$ seems yet established in the literature. In some mildly sparse or mildly dense cases (i.e., $(\delta,r)=(0.3,4), (0.25,6)$), we see that the adaptive test based on $q=(2,4)$ slightly outperforms $L_2\&L_{\infty}$, demonstrating the additional merit by incorporating $L_q$ norm based test, where $q\not=2, \infty$, in forming the adaptive test. 

%type applied  to combine  statistics in the existing literature, and have similar performance to our adaptive method in the numerical studies, the proposed test is based on asymptotic independence, which is the first result in the literature for general kernels. 

\subsection{Tests for Linear Model Coefficients}
\label{sec::sim_lr}

In this subsection, we present some simulation results for testing the nullity of linear regression coefficients. Let the data $\{(X_i,Y_i)\}_{i=1}^n$ be i.i.d. samples from linear model $Y_i=X_i\bbeta+\varepsilon_i$ and the goal is to test $\cH_0:\bbeta=\bm{0}$ against $\cH_a:\bbeta\not=\bm{0}$. Recall that the kernel function is of order 2 as in (\ref{lrhl}), with $\bbeta_0=\bm{0}$,
$h_l(W_1,W_2)=(x_{1,l}-x_{2,l})(Y_1-Y_2)/2,$
and the parameter to test equals $\Sigma \bbeta$. The test statistic proposed in \cite{zhong2011tests} is indeed a special case of ours with $q=2$ (hence their kernel is of order $qr$=4), standardized by a different variance estimator.  We set $(n,p)=(100,50),\,(200,100)$, $\bbeta=\delta (\bm{1}_r,\bm{0}_{p-r})$, and generate $X_i\sim N(0,\Sigma)$ independent of $\varepsilon_i\sim N(0,1)$. We consider $\Sigma=(\sigma_{a,b})_{a,b=1}^p$ from AR(1) structure with $\sigma_{a,b}=\rho^{|a-b|}$, and $q=2,4,6$.

In Table~\ref{tb:linearmodel}, we present the results for $U_{n,2}$, $U_{n,q}^M$-based tests (with permutation-based variance estimator) for $q=2,4,6, (2,4)$ and $(2,6)$, Bonferonni type combination of $U_{n,2}$ with $U_{n,q}^M$ for $q=4,6$. 
Computationally speaking, the computational cost for the monotonically indexed  U-statistic $U_{n,q}^M$ and the studentized statistic $T_{n,q}^M$  is $O(n^2)$ for any $q$. The full calculation of $U_{n,2}$ is implemented via the efficient algorithm proposed in \cite{leung2018testing} so the cost is reduced from $O(n^4)$ to $O(n^2)$. In our unreported simulations,  we also compare the performance of $U_{n,2}$ (with permutation-based variance estimator) to that of \cite{zhong2011tests}, 
%which is a special case of the proposed U-statistic $U_{n,q}$ with $q=2$ but with a different variance estimator, 
and find that they lead to very similar results. Note that the computation of the studentized test in \cite{zhong2011tests} requires $O(n^4)$ computation.

%so we   that when $q=2$, the permutation-based variance estimator and  the variance estimator used in \cite{zhong2011tests} for $U_{n,2}$ result in very similar results,  but the former is much faster to compute as the latter requires $O(n^4)$ computation. 
%To manage the computational time, we run 1000 Monte Carlo repetitions for the size simulation, and 500 for power simulation.

%Therefore, we will focus on the permutation-based variance estimator for $U_{n,2}$ in our numerical study, instead of following \cite{zhong2011tests}. By applying the same approach of variance approximation, we can also compare directly the power of $U_{n,2}$ and $U_{n,2}^M$. 

%The first test in the high-dimensional setting  was proposed by \cite{zhong2011tests} via U-statistic of order $4$. testing studied in \cite{zhong2011tests}. Note that in particular, 

%We use the permutation based method to estimate the asymptotic variance introduced in Section \ref{sec::var_est}. We randomly permute the pairings between $X_i$ and $Y_i$, and estimate the variance based on the statistics calculated from $B=100$ permutations. We can also use these 100 permutations to conduct a permutation test, and the results are very similar and hence not presented. 

\begin{table}[H]
\centering
\begin{tabular}{|c|c|c|c|c|c|c|c|c|c|}
\hline
\multirow{2}{*}{$(n,p)$} & \multirow{2}{*}{$\delta,r,\rho$}  & $U_{n,q}$ & \multicolumn{5}{c|}{$U_{n,q}^M$} &\multicolumn{2}{c|}{$U_{n,2}$ \& $U_{n,q}^M$}
\\ \cline{3-10}
& & \multicolumn{2}{c|}{$q=2$} & $q=4$ & $q=6$ & $q=2,4$ & $q=2,6$ & $q=4$ & $q=6$  \\ \hline
\multirow{6}{*}{(100,50)} & 0, NA, 0 & 6.3 & 5.7 & 6.4 & 3.1 & 6.3 & 4.6 & 7.2 & 5.7 \\ \cline{2-10}
& 0, NA, 0.5 & 7.0 & 6.5 & 4.6 & 4.1 & 6.5 & 5.7 & 7.0 & 6.4
\\ \cline{2-10}
& 0.25, 2, 0 & 29.2 & 21.2 & 28.4 & 17.8 & 31.2 & 25.4 & 34.4 & 29.4 
\\ \cline{2-10}
& 0.25, 2, 0.5 & 60.4 & 45.6 & 62.6 & 49.6 & 65.2 & 57.0 & 69.6 & 62.8
\\ \cline{2-10}
& 0.05, $p$, 0 & 31.2 & 24.8 & 8.0 & 6.8 & 22.8 & 22.2 & 27.8 & 26.4 \\ \cline{2-10}
& 0.05, $p$, 0.5 & 98.0 & 92.2 & 45.6 & 19.6 & 89.6 & 89.6 & 96.8 & 97.0 \\ \hline
\multirow{6}{*}{(200,100)} & 0, NA, 0 & 5.7 & 5.5 & 4.9 & 3.5 & 5.1 & 5.6 & 5.2 & 5.3 \\ \cline{2-10}
& 0, NA, 0.5 & 6.6 & 5.9 & 3.9 & 2.6 & 6.3 & 5.1  & 6.8 & 5.5 \\ \cline{2-10}
& 0.25, 2, 0 & 45.2 & 32.4 & 56.8 & 46.8 & 58.2 & 52.6 & 63.8 & 57.8 \\ \cline{2-10}
& 0.25, 2, 0.5 & 81.6 & 62.2 & 93.6 & 90.4 & 92.8 & 92.0 & 94.6 & 93.8 \\ \cline{2-10}
& 0.035, $p$, 0 & 42.4 & 30.4 & 10.2 & 7.6 & 27.8 & 28.0 & 36.2 & 36.6
\\ \cline{2-10}
& 0.035, $p$, 0.5 & 100 & 99.4 & 57.6 & 28.6 & 99.4 & 99.4 & 100 & 100\\  \hline
\end{tabular}
\caption{Size and power in percentage for linear model coefficients  testing}
\label{tb:linearmodel}
\end{table}

As we can see from Table~\ref{tb:linearmodel}, all tests appear to have quite reasonable size, and there are some mild size distortion in some cases. As expected, large $q$ has higher power against sparse alternative while small $q$ outperforms under dense alternative. Comparing $U_{n,2}$ with $U_{n,2}^M$, we see a noticeable power loss for $U_{n,2}^M$. This is not surprising since the monotonically indexed U-statistic tends to lose some statistical efficiency. In the event that the computational complexity of $L_2$-norm test statistic can be made to $O(n^2)$, we recommend to calculate the full U-statistic, i.e., $U_{n,2}$ and combine with $U_{n,q}^M$ for $q=4,6$. For this example, the adaptive test based on $U_{n,2}$ and $U_{n,4}^{M}$ appears  to have the best power or close to the best single-$q$ test in power. 

% The adaptive test always has comparable performance to the best single-$q$ based test under all settings. Note that $q=4$ tends to have high power than $q=6$, mainly due to the conservative size of $q=6$ under the null.

%Under sparse alternative, the ZC test has comparable power to the proposed adaptive test with $q=(2,4)$ or (2,6). Under dense alternative, we observe obvious power loss of the ZC test compared to our adaptive test. Moreover, \cite{zhong2011tests} uses a U-statistic of order 4, and requires $O(n^4)$ computation, while our method only needs $O(n^2)$.

\subsection{Two-Sample Tests for Spatial Sign}
\label{sec::sim_ts_spatial}

In this subsection, we study the size and power property for two-sample spatial sign test developed by \citep{chakraborty2017tests} and compare with its $L_q$-norm based extension and adaptive tests.
%We generate $\{X_i\}_{i=1}^n,\{Z_j\}_{j=1}^m$ from some centered-distributions including multivariate normal and $t_3$ distributions. 
The first sample is generated as $\{X_i\}_{i=1}^n=\{Z_i+\mu\}_{i=1}^n$ and the second sample $\{Y_j\}_{j=1}^m$, where  we set $\mu=\delta(\bm{1}_{r}^T,\bm{0}_{p-r}^T)^T,$ and generate $\{Y_i\}$ and $\{Z_j\}$ in the following way:
\begin{itemize}
    \item Standard Gaussian: $Y_i,Z_j\iid N(0,I_p);~i=1,\ldots,n;\,j=1,\ldots,m$.
    \item $t$-distribution: $Y_{il},Z_{jl}\iid  t_3;~i=1,\ldots,n;\,j=1,\ldots,m;\,l=1,\ldots,p$.
\end{itemize}
We set $n=m=50,\,p=100$. 

\begin{table}[H]
\centering
\begin{tabular}{|c|c|c|c|c|c|c|}
\hline
DGP & $\delta,r$ & $q=2$ & $q=4$ & $q=6$ & $q=2,4$ & $q=2,6$\\ \hline
\multirow{3}{*}{Gaussian} & 0, NA & 2.9 & 6.2 & 6.8 & 4.2 & 5.2\\ \cline{2-7}
& 1, 2  & 17.2 & 94.2 & 91.0 & 91.2 & 90.8\\ \cline{2-7}
& 0.3, $p$ & 96.4 & 69.2 & 37.2 & 95.6 & 96.2 \\  \hline
\multirow{3}{*}{$t_3$} & 0, NA & 2.4 & 5.0 & 5.6 & 4.2 & 4.7\\ \cline{2-7}
& 1, 2 & 3.8 & 37.8 & 41.4 & 34.0 & 38.6\\ \cline{2-7}
& 0.5, $p$ & 98.3 & 76.6 & 45.6 & 97.4 & 97.5\\  \hline
\end{tabular}
\caption{Size and power in percentage for two-sample spatial sign test}
\label{tb:spatialsign2}
\end{table}
According to Table~\ref{tb:spatialsign2},  most tests have relatively accurate size, except for the case $q=2$, which exhibits  slightly conservative size for two-sample spatial sign test. This is opposite to the findings in one-sample test where large $q$ tends to be conservative. As is the case for one-sample spatial sign test, small and large $q$ correspond to high power against dense and sparse alternatives, respectively. Again, the adaptive test based on $(2,4)$ or $(2,6)$ combines the advantages of single $q$ tests  and has comparable power to the best single-$q$ test. 

\section{Real Data Application}
\label{sec:data}

In this section, we apply the two-sample spatial sign based test (with permutation) to the Connectionist Bench dataset publicly available at machine learning UCI website (\url{http:
//www.cs.ucr.edu/eamonn/time_series_data}).  The dataset we examined contains 111 samples of sonar signals from metal cylinder and 97 samples from rocks. We follow the approach in \citep{chakraborty2017tests} to study the size and power of the proposed two sample test with different $q$s. For the size simulation, we randomly generate two subsamples with sample size 40 from the rock sample, and calculate the average rejection rate. In terms of the power, we generate two subsamples also with sample size 40 from metal and rock samples, and find the average rejection probability from 1000 repetitions. The results are summarized in Table \ref{tb:real}.
\label{sec:apply}
\begin{table}[H]
\centering
\begin{tabular}{|c|c|c|c|c|c|}
\hline
Statistic & $q=2$ & $q=4$ & $q=6$ & $q=2,4$ & $q=2,6$\\ \hline
Size  & 2.1 & 5.9 & 6.2 & 4.1 & 5.5\\ \hline
Power & 22.2 & 48.7 & 30.1 & 42.6 & 33.2\\ \hline
\end{tabular}
\caption{Size and power in percentage for a real data application}
\label{tb:real}
\end{table}
As we can see from the table, $q=2$ exhibits conservative size while all other tests have relatively accurate size. As regards the power, $q=4$ has highest power followed by $q=6$, which outperforms $q=2$, implying that the underlying alternative is better captured with $q=4$ than $q=2$. With adaptive test combining $q=2,4$ (or $2,6$), we can obtain higher power than the single $q$ based statistic with $q=2$ as studied in \cite{chakraborty2017tests}. This clearly demonstrates the value of 
the adaptive tests we developed. In practice, it would be informative to determine which $q$ is more powerful to the data at hand, and which pair of $q$s to use in adaptive combination. We leave this topic to future research.

\section{Conclusion}
\label{sec:conclude}
In this paper, we propose a family of $L_q$-norm based U-statistics for high-dimensional i.i.d. data and show that they are asymptotically normal and independent for different $q$s under mild moment and cumulant assumptions. Our work is inspired by \cite{he2021asymptotically}, who first demonstrated the asymptotic normality and independence of $L_q$-norm based U-statistics for testing of mean vectors and covariance matrices. Our results are obtained for a broad class of kernel functions, and can be viewed as a significant extension of those in \cite{he2021asymptotically}, which is limited to kernel function of order $1$. 
In particular, we highlight the usefulness of our results in both theory and simulations for  several global testing problems that are not covered by \cite{he2021asymptotically}, including one sample spatial sign based test, simultaneous testing of linear model coefficients, component-wise independent tests via Kendall's $\tau$, and two sample spatial sign based test. 

To capture both sparse and dense alternatives, we develop a simple Bonferonni type combination of two studentized $L_q$-norm based test statistics via their p-values, and show that it leads to an adaptive test with high power against both dense and sparse alternatives. To alleviate the computational burden, we propose a variant of symmetric U-statistic by focusing on monotone indices in the summation, to which the idea of dynamic programming can be applied. With some moderate loss of statistical efficiency, we are able to reduce the computation cost from $O(n^{qr})$, which corresponds to full symmetric U-statistic,  to $O(n^{r})$, where $r$ is the order of the kernel.
According to the simulation results, it appears that a combination of the existing $L_2$ norm based test with monotonically indexed $L_q$ norm based test ($q=4$ or $6$) achieves satisfactory power against both sparse and dense alternatives with the computational cost of order $O(n^2)$. 
%We also propose a permutation based variance estimator for $U_{n,q}^M$ to construct the standardized statistic.

%The proposed statistics are built up under a very general setting, and is applicable to many interesting testing problems, including tests for spatial sign, linear model, component-wise independence. They also include many existing literature studying these problems as special cases. 
%We also provide an asymptotic distribution under the local alternative, which is a non-centered normal distributions. Based on that, we compare the asymptotic power with different $q$. We find that the optimal $q$ whose associated $U_{n,q}$ has highest power depends on the sparsity of the alternatives, which is consistent with the findings in \cite{he2021asymptotically} for mean and covariance testing. Therefore, a straightforward combination of $U_{n,q}$ with different $q$'s lead to an adaptive test with high power against both dense and sparse alternatives. Lastly, we propose a variant U-statistic $U_{n,q}^M$ with monotone indices, which the idea of dynamic programming can be applied to so that the computation in $n$ is reduced from $O(n^{qr})$ to $O(n^{r})$. We also propose a permutation based variance estimator for $U_{n,q}^M$ to construct the standardized statistic.

To conclude, we mention a few worthy future directions.  It would be interesting to generalize the asymptotic theory to high-dimensional time series, as many of the global testing problems we tackled here can be naturally posed for high-dimensional time series data. We expect that this extension will be nontrivial since temporal dependence can bring non-negligible impact on the validity of normal approximation for $L_q$-norm based test statistic.  
Moreover, it might also be interesting to extend our methodology and theory to change point detection problem, which is similar to but more involved than a two sample testing problem due to the unknown change-point locations; see \cite{zhang2021adaptive} for the use of $L_q$-norm based test statistics for the mean change-point testing in high-dimensional data.  Finally, estimation of asymptotic variance is an important problem for inference and yet to be addressed in the high-dimensional setting. We conjecture that Bootstrap can work but may deserve some caution due to the  high-dimensionality of observations, and a rigorous theory is needed. We leave these open problems for future study.

%is also an interesting but difficult problem to study the theoretical properties of Bootstrap approximation for $U_{n,q}^M$, as an alternative approach for approximating the null distribution. 
\begin{supplement}
	%\sname{Supplement A}\label{suppA}
	\stitle{Supplement to ``Adaptive Testing for High-dimensional Data"}
	%\slink[url]{}
	\sdescription{The supplementary material contains all the proofs for theoretical results stated in the paper.}
\end{supplement}

\bibliographystyle{apalike}
\bibliography{ref}

\begin{thebibliography}{}

\bibitem[Andrews, 1991]{andrews1991heteroskedasticity}
Andrews, D.~W. (1991).
\newblock Heteroskedasticity and autocorrelation consistent covariance matrix
  estimation.
\newblock {\em Econometrica: Journal of the Econometric Society}, pages
  817--858.

\bibitem[Arcones and Gine, 1992]{arcones1992bootstrap}
Arcones, M.~A. and Gine, E. (1992).
\newblock On the bootstrap of {U} and {V} statistics.
\newblock {\em The Annals of Statistics}, 20(2):655--674.

\bibitem[Bai et~al., 2009]{bai2009corrections}
Bai, Z., Jiang, D., Yao, J.-F., and Zheng, S. (2009).
\newblock Corrections to lrt on large-dimensional covariance matrix by rmt.
\newblock {\em The Annals of Statistics}, 37(6B):3822--3840.

\bibitem[Bai and Saranadasa, 1996]{bai1996effect}
Bai, Z. and Saranadasa, H. (1996).
\newblock Effect of high dimension: by an example of a two sample problem.
\newblock {\em Statistica Sinica}, 6(2):311--329.

\bibitem[Cai et~al., 2013]{cai2013two}
Cai, T., Liu, W., and Xia, Y. (2013).
\newblock Two-sample covariance matrix testing and support recovery in
  high-dimensional and sparse settings.
\newblock {\em Journal of the American Statistical Association},
  108(501):265--277.

\bibitem[Cai and Jiang, 2011]{cai2011limiting}
Cai, T.~T. and Jiang, T. (2011).
\newblock Limiting laws of coherence of random matrices with applications to
  testing covariance structure and construction of compressed sensing matrices.
\newblock {\em The Annals of Statistics}, 39(3):1496--1525.

\bibitem[Cai et~al., 2014]{cai2014two}
Cai, T.~T., Liu, W., and Xia, Y. (2014).
\newblock Two-sample test of high dimensional means under dependence.
\newblock {\em Journal of the Royal Statistical Society: Series B: Statistical
  Methodology}, 76(2):349--372.

\bibitem[Chakraborty and Chaudhuri, 2017]{chakraborty2017tests}
Chakraborty, A. and Chaudhuri, P. (2017).
\newblock Tests for high-dimensional data based on means, spatial signs and
  spatial ranks.
\newblock {\em The Annals of Statistics}, 45(2):771--799.

\bibitem[Chen and Qin, 2010]{chenqin2010}
Chen, S.~X. and Qin, Y.-L. (2010).
\newblock A two-sample test for high-dimensional data with applications to
  gene-set testing.
\newblock {\em The Annals of Statistics}, 38(2):808--835.

\bibitem[Chen et~al., 2010]{chen2010tests}
Chen, S.~X., Zhang, L.-X., and Zhong, P.-S. (2010).
\newblock Tests for high-dimensional covariance matrices.
\newblock {\em Journal of the American Statistical Association},
  105(490):810--819.

\bibitem[Drton et~al., 2020]{drton2020high}
Drton, M., Han, F., and Shi, H. (2020).
\newblock High-dimensional consistent independence testing with maxima of rank
  correlations.
\newblock {\em The Annals of Statistics}, 48(6):3206--3227.

\bibitem[El~Karoui, 2009]{el2009concentration}
El~Karoui, N. (2009).
\newblock Concentration of measure and spectra of random matrices: Applications
  to correlation matrices, elliptical distributions and beyond.
\newblock {\em The Annals of Applied Probability}, 19(6):2362--2405.

\bibitem[Goeman et~al., 2006]{goeman2006testing}
Goeman, J.~J., Van De~Geer, S.~A., and Van~Houwelingen, H.~C. (2006).
\newblock Testing against a high dimensional alternative.
\newblock {\em Journal of the Royal Statistical Society: Series B (Statistical
  Methodology)}, 68(3):477--493.

\bibitem[Gregory et~al., 2015]{gregory2015two}
Gregory, K.~B., Carroll, R.~J., Baladandayuthapani, V., and Lahiri, S.~N.
  (2015).
\newblock A two-sample test for equality of means in high dimension.
\newblock {\em Journal of the American Statistical Association},
  110(510):837--849.

\bibitem[Hall and Jin, 2010]{hall2010innovated}
Hall, P. and Jin, J. (2010).
\newblock Innovated higher criticism for detecting sparse signals in correlated
  noise.
\newblock {\em The Annals of Statistics}, 38(3):1686--1732.

\bibitem[Han et~al., 2017]{han2017distribution}
Han, F., Chen, S., and Liu, H. (2017).
\newblock Distribution-free tests of independence in high dimensions.
\newblock {\em Biometrika}, 104(4):813--828.

\bibitem[He et~al., 2021]{he2021asymptotically}
He, Y., Xu, G., Wu, C., and Pan, W. (2021).
\newblock Asymptotically independent u-statistics in high-dimensional testing.
\newblock {\em The Annals of Statistics}, 49(1):154--181.

\bibitem[Jiang, 2004]{jiang2004asymptotic}
Jiang, T. (2004).
\newblock The asymptotic distributions of the largest entries of sample
  correlation matrices.
\newblock {\em The Annals of Applied Probability}, 14(2):865--880.

\bibitem[Ledoit and Wolf, 2002]{ledoit2002some}
Ledoit, O. and Wolf, M. (2002).
\newblock Some hypothesis tests for the covariance matrix when the dimension is
  large compared to the sample size.
\newblock {\em The Annals of Statistics}, 30(4):1081--1102.

\bibitem[Leung and Drton, 2018]{leung2018testing}
Leung, D. and Drton, M. (2018).
\newblock Testing independence in high dimensions with sums of rank
  correlations.
\newblock {\em The Annals of Statistics}, 46(1):280--307.

\bibitem[Li and Chen, 2012]{li2012two}
Li, J. and Chen, S.~X. (2012).
\newblock Two sample tests for high-dimensional covariance matrices.
\newblock {\em The Annals of Statistics}, 40(2):908--940.

\bibitem[Liu et~al., 2008]{liu2008asymptotic}
Liu, W.-D., Lin, Z., and Shao, Q.-M. (2008).
\newblock The asymptotic distribution and berry--esseen bound of a new test for
  independence in high dimension with an application to stochastic
  optimization.
\newblock {\em The Annals of Applied Probability}, 18(6):2337--2366.

\bibitem[Shao and Zhou, 2014]{shao2014necessary}
Shao, Q.-M. and Zhou, W.-X. (2014).
\newblock Necessary and sufficient conditions for the asymptotic distributions
  of coherence of ultra-high dimensional random matrices.
\newblock {\em The Annals of Probability}, 42(2):623--648.

\bibitem[Shao and Wu, 2007]{shao2007local}
Shao, X. and Wu, W.~B. (2007).
\newblock Local whittle estimation of fractional integration for nonlinear
  processes.
\newblock {\em Econometric Theory}, 23(5):899--929.

\bibitem[Srivastava and Du, 2008]{srivastava2008test}
Srivastava, M.~S. and Du, M. (2008).
\newblock A test for the mean vector with fewer observations than the
  dimension.
\newblock {\em Journal of Multivariate Analysis}, 99(3):386--402.

\bibitem[Srivastava et~al., 2016]{srivastava2016raptt}
Srivastava, R., Li, P., and Ruppert, D. (2016).
\newblock Raptt: An exact two-sample test in high dimensions using random
  projections.
\newblock {\em Journal of Computational and Graphical Statistics},
  25(3):954--970.

\bibitem[Wang et~al., 2015]{wang2015high}
Wang, L., Peng, B., and Li, R. (2015).
\newblock A high-dimensional nonparametric multivariate test for mean vector.
\newblock {\em Journal of the American Statistical Association},
  110(512):1658--1669.

\bibitem[Wang et~al., 2022]{wang2022inference}
Wang, R., Zhu, C., Volgushev, S., and Shao, X. (2022).
\newblock Inference for change points in high-dimensional data via
  selfnormalization.
\newblock {\em The Annals of Statistics}, 50(2):781--806.

\bibitem[Wu et~al., 2019]{wu2019adaptive}
Wu, C., Xu, G., and Pan, W. (2019).
\newblock An adaptive test on high-dimensional parameters in generalized linear
  models.
\newblock {\em Statistica Sinica}, 29(4):2163--2186.

\bibitem[Wu, 2005]{wu2005nonlinear}
Wu, W.~B. (2005).
\newblock Nonlinear system theory: Another look at dependence.
\newblock {\em Proceedings of the National Academy of Sciences},
  102(40):14150--14154.

\bibitem[Wu and Shao, 2004]{wu2004limit}
Wu, W.~B. and Shao, X. (2004).
\newblock Limit theorems for iterated random functions.
\newblock {\em Journal of Applied Probability}, 41(2):425--436.

\bibitem[Xu et~al., 2016]{xu2016adaptive}
Xu, G., Lin, L., Wei, P., and Pan, W. (2016).
\newblock An adaptive two-sample test for high-dimensional means.
\newblock {\em Biometrika}, 103(3):609--624.

\bibitem[Zhang et~al., 2021]{zhang2021adaptive}
Zhang, Y., Wang, R., and Shao, X. (2021).
\newblock Adaptive inference for change points in high-dimensional data.
\newblock {\em Journal of the American Statistical Association}, In press.

\bibitem[Zhong and Chen, 2011]{zhong2011tests}
Zhong, P.-S. and Chen, S.~X. (2011).
\newblock Tests for high-dimensional regression coefficients with factorial
  designs.
\newblock {\em Journal of the American Statistical Association},
  106(493):260--274.

\bibitem[Zhurbenko and Zuev, 1975]{zhurbenko1975higher}
Zhurbenko, I. and Zuev, N. (1975).
\newblock On higher spectral densities of stationary processes with mixing.
\newblock {\em Ukrainian Mathematical Journal}, 27(4):364--373.

\end{thebibliography}

\newpage
\appendix
\bigskip
\begin{center}
{\large\bf SUPPLEMENT TO ``ADAPTIVE TESTING FOR HIGH-DIMENSIONAL DATA''}
\end{center}

\begin{center}
{BY YANGFAN ZHANG, RUNMIN WANG AND XIAOFENG SHAO}
\end{center}
The supplement is organized as follows. In Appendix \ref{sec:verify}, we verify the assumptions for the main theory for several testing problems we focus on. Appendix \ref{sec:tech} includes all technical proofs.

\section{Verification of Assumptions}
\label{sec:verify}
In this section, we verify the assumptions on different testing problems. We first summarize the assumptions we need to check. Assumption \ref{ass::lead_hoeff} is trivial if the order of the kernels is one ($r=1$), or more generally, if the kernel is fully degenerate ($r=s$). Assumption \ref{ass::asy_deg} (b2) is necessary only for degenerate kernels ($s>1$). 

\subsection{Mean testing and covariance testing}
\label{sec:verify_mean}
These two examples have kernels with order $r=1$, so we only meed to check Assumption \ref{ass::asy_deg} (a) and (b1). They are implied by Assumption \ref{ass::cumr} and the proof can be found in \cite{wang2022inference} for mean testing, with $C_p=c_0$ for some constant $c_0$. Therefore, our theory applies if $X_i$ satisfies Assumption \ref{ass::cumr}. In terms of covariance testing, it suffices to view $(x_{i,p_1}x_{i,p_2})_{(p_1,p_2)\in\cL}$ as i.i.d. random vectors in $\R^{|\cL|}$ and verify Assumption \ref{ass::cumr} for them.

\subsection{Spatial sign testing}
\label{sec:verify_spatial}
For spatial sign, recall $h_l(X_i)=x_{i,l}/\|X_i\|$ and $H_{l,q}(X_1,\ldots,X_q)=h_l(X_1)\cdots h_l(X_q)$. It is technically difficult to verify Assumption \ref{ass::cumr} for $h_l$ directly. Instead, we show the martingale CLT for $D_{n,q}$. As we need to go through the proof of the main theorem in Appendix \ref{sec:tech}, we only highlight the new argument required for spatial sign. Recall that the associated martingale difference sequence (with proper normalization) is given by
\begin{align*}
% \xi_{n,i}=&\E[\alpha_1n^{-q/2}\tsig^{-1/2}_1(q)D_{n,q}+\alpha_2n^{-Q/2}\tsig^{-1/2}_1(Q)D_{n,Q}\mid\F_{i-1}]\\
 \xi_{n,i}=&\alpha_1 q!n^{-q/2}\tsig^{-1/2}_1(q)\sum_{l\in\cL}\sum_{1\le i_1<\cdots<i_{q-1}< i}H_{l,q}^{(t)}(X_{i_1},\ldots, X_{i_{q-1}},X_i)\\
&+\alpha_2Q!n^{-Q/2}\tsig^{-1/2}_1(Q)\sum_{l\in\cL}\sum_{1\le i_1<\cdots<i_{Q-1}< i}H_{l,Q}^{(t)}(X_{i_1},\ldots, X_{i_{Q-1}},X_i)\\
\deltaeq&\alpha_1q!\xi_{n,i}^{(1)}+\alpha_2Q!\xi_{n,i}^{(2)}.
\end{align*}

Following \cite{wang2015high}, we consider elliptically distributed random vectors $X_i$ with
\begin{equation}
X_i=\mu+\varepsilon_i,\quad\varepsilon_i=\Gamma R_iU_i,
\label{eq::elldist}
\end{equation}
where $\Gamma\in\R^{p\times p}$, $U_i$ is uniformly distributed on the unit sphere in $\R^p$ and $R_i\ge 0$ is some scale random variable independent with $U_i$. Denote $\Omega=\Gamma\Gamma^T$, then the covariance matrix $\Sigma=p^{-1}D\Omega$ with $D\deltaeq\E[R_i^2]$. Denote $W_i=\Gamma U_i=(w_{i,1},\ldots,w_{i,p})^T$ and we have $X_i=R_iW_i$ so $\E[W_i]=\E[X_i]=\bz$ under the null .

Consider the event $A_i=\left\{ 
\|W_i\|^2\ge \frac{\tr(\Omega)}{2p}\right\}$.  From Lemma 3 in \cite{wang2015high}, we have $$P(A_i)\ge 1-c_1\exp\left\{\frac{\tr^2(\Omega)}{-128p\lambda_{\max}^2(\Omega)}\right\},$$ 
with $c_1=2\exp{(\pi/2)}$ being a constant.

We make the following assumption similar to Assumption (C2) in \cite{wang2015high}, so that we may apply the concentration inequality developed by \citep{el2009concentration}.
\begin{assumption}
Assume the random vectors $X_i$ follow elliptical distribution as \eqref{eq::elldist}, and
\begin{equation}
p^{2q-2}\exp\left\{\frac{\tr^2(\Omega)}{-128p\lambda_{\max}^2(\Omega)}\right\}=O(1).
\label{eq::concen_bound}
\end{equation}
%Moreover, $\E[R_i^u]\le C D^{u/2}$ for some constant $C$ and $u=3,4$.
\label{ass::concen}
\end{assumption}
The assumption is mildly stronger than \cite{wang2015high} and can be verified similar to the argument therein, since we expect the exponential term to decay to 0 at a fast rate. To be more precise, let $\lambda_1\le\lambda_2\le\ldots\le\lambda_p$ denote the eigenvalues of $\Sigma$. Assume there are $k_1$ eigenvalues decaying to 0 and $k_2$ ones diverging to $\infty$ at rate $p^a$ for some $a>0$. The remaining $(p-k_1-k_2)$ eigenvalues are assumed to be bounded below by some $b>0$. We assume $k_1$ and $k_2$ are bounded and recall that we consider high-dimensional setting where $p\rightarrow\infty$.
We have 
$$
\frac{\tr(\Sigma)}{\sqrt{p}\lambda_{\max}(\Sigma)}\ge \frac{b(p-k_1-k_2)}{\sqrt{p}\lambda_p}.
$$
Therefore, the exponential term decays to 0 faster than $p^{2-2q}$ as long as $a<1/2$.
The assumptions can also be satisfied with extra restrictions even if $k_1$ and $k_2$ also diverge to $\infty$. The second statement essentially assumes that $R_i$ has bounded fourth moment after proper standardization. 

To apply the martingale CLT (Theorem 35.12 in Billingsley (2008)), we need to verify the following two conditions.\\
$\displaystyle(i)\quad\forall\varepsilon>0,\sum_{i=1}^n\mathbb{E}\left[{\xi}_{n,i}^{2} \textbf{1}\left\{\left|{\xi}_{n,i}\right|>\varepsilon\right\} \Big| \mathcal{F}_{i-1}\right] \stackrel{p}{\rightarrow} 0$.\\
$\displaystyle(ii)\quad V_n=\sum_{i=1}^{ n}\mathbb{E}\left[{\xi}_{n,i}^{2}|\F_{i-1}\right] \stackrel{p}{\rightarrow} \alpha_1^2(qr)!+\alpha_2^2(Qr)!$.

The key result required for (i) (ii) is equation (14) or Lemma 7.2 (i), i.e.,
\begin{equation}
\sum_{l_1,\ldots,l_4 \in\cL}\E\left[\prod_{u=1}^4 h_{l_u}(X_{i_1^{(u)}})\cdots h_{l_u}(X_{i_q^{(u)}})\right]=\sum_{l_1,\ldots,l_4 \in\cL}\E\left[\prod_{u=1}^4 H_{l_u}(X_{i_1^{(u)}},\cdots,X_{i_q^{(u)}})\right]\le \tsig_1^2(q).
\label{eq:spatial_prod}
\end{equation}
Consider the event $A_q=\left\{ 
\|W_i\|^2\ge \frac{\tr(\Omega)}{2p},\,i=1,\ldots,q.\right\}$.  From Lemma 3 in \cite{wang2015high}, we have $$P(A_q)\ge 1-qc_1\exp\left\{\frac{\tr^2(\Omega)}{-128p\lambda_{\max}^2(\Omega)}\right\},$$ 
with $c_1=2\exp{(\pi/2)}$ being a constant.

Note that we have
\begin{align*}
&\sum_{l_1,\ldots,l_4 \in\cL}\E\left[\prod_{u=1}^4 H_{l_u}(X_{i_1^{(u)}},\cdots,X_{i_q^{(u)}})\right]\\
\le&\E\left[\left(\sum_{l\in\cL}H_l(X_1,\cdots,X_q)\right)^4\right]\\
\le & \E\left[\left(\frac{\sum_{l\in\cL}w_{1,l}\cdots w_{q,l}}{\|W_1\|\cdots\|W_q\|}\right)^4(1_{A_q}+1_{A_q^c})\right]\\
\le & (\tr(\Omega)/2p)^{-2q}\E\left[\left(\sum_{l\in\cL}w_{1,l}\cdots w_{q,l}\right)^4\right]+P(A_q^c)\quad (q\ge 2).\\
\le & (\tr(\Omega)/2p)^{-2q}(\E[R_1^4])^{-q}\E\left[\left(\sum_{l\in\cL}x_{1,l}\cdots x_{q,l}\right)^4\right]+P(A_q^c)\quad (x_{i,l}=R_iw_{i,l}).\\
\lesssim&(D\tr(\Omega)/2p)^{-2q}\E\left[\left(\sum_{l\in\cL}x_{1,l}\cdots x_{q,l}\right)^4\right]+P(A_q^c)\quad (\E[R_i^4]\ge \E^2[R_i^2]=D^2)\\
\lesssim&\tr^{-2q}(\Sigma)\E\left[\left(\sum_{l\in\cL}x_{1,l}\cdots x_{q,l}\right)^4\right]+P(A_q^c)\quad (\Sigma=p^{-1}D\Omega)
\end{align*}
Note that we have (from the verification for mean-testing)
$$
\E\left[\left(\sum_{l\in\cL}x_{1,l}\cdots x_{q,l}\right)^4\right]\lesssim \|\Sigma\|_q^{2q},
$$
is of order $p^2$ at most which along with the fact that $\tr(\Sigma)$ is of order at least $p$ implies 
$$
\tr^{-2q}(\Sigma)\E\left[\left(\sum_{l\in\cL}x_{1,l}\cdots x_{q,l}\right)^4\right]\lesssim p^{2-2q}.
$$
As regards $\tsig_1(q)$, we have
\begin{align*}
\tsig_1^{1/q}(q)=&\left(\sum_{l_1,l_2\in\cL}\E^q\left[\frac{x_{1,l_1}x_{1,l_2}}{\|X_1\|^2}\right]\right)^{1/q}\ge\left(\sum_{l\in\cL}\E^q\left[\frac{x_{1,l}^2}{\|X_1\|^2}\right]\right)^{1/q}\\
=&p^{1/q-1}\left(\sum_{l\in\cL}\E^q\left[\frac{x_{1,l}^2}{\|X_1\|^2}\right]\right)^{1/q}\left(\sum_{l\in\cL}1^{q/(q-1)}\right)^{1-1/q}\\
\ge& p^{1/q-1}\sum_{l\in\cL}\E\left[\frac{x_{1,l}^2}{\|X_1\|^2}\right]=p^{1/q-1}.
\end{align*}
Therefore, we conclude that $\tsig_1(q)\ge p^{1-q}$. To conclude \eqref{eq:spatial_prod}, it remains to note that \eqref{eq::concen_bound} implies $P(A_q^c)\lesssim p^{2-2q}$. The proof is then complete.

\subsection{Testing the nullity of linear regression coefficients}
\label{sec:verify_linear}
Denote $W_i=(X_i,Y_i)$. Without loss of generality, we assume $\E[X]=0$ and let $\Sigma=(\Sigma_{l_1,l_2})_{l_1,l_2=1}^p\deltaeq\var(X)$. Under the null, we have 
$$h_l(W_{[2]})=(\varepsilon_1-\varepsilon_2)(x_{1,l}-x_{2,l}),\quad l=1,\ldots,p,$$
where we have dropped the factor 1/2 to ease the notations.
Therefore,
$$
h_l^{(1)}(W_1)=\E[h_l(W_{[2]})|W_1]=\varepsilon_1x_{1,l},
$$
and 
$$h_l^{(2)}(W_{[2]})=h_l(W_{[2]})-h_l^{(1)}(W_1)-h_l^{(1)}(W_2)=-(\varepsilon_1x_{2,l}+\varepsilon_2x_{1,l}).$$
Since the kernel is not degenerate, it suffices to check Assumption \ref{ass::lead_hoeff} and \ref{ass::asy_deg} (a,b1). However, by defining i.i.d. zero-mean random variables $Y_i\deltaeq\varepsilon_iX_i$ such that $\var(Y)=\sigma^2\cdot \Sigma$ with $\sigma^2\deltaeq\var(\varepsilon_1)$, we can see that $h_l^{(1)}(W_i)=y_{i,l}$. Therefore, Assumption \ref{ass::asy_deg} (a) and (b1) are satisfied if Assumption \ref{ass::cumr} holds for i.i.d. random vectors $\varepsilon_iX_i$.

Next we check Assumption \ref{ass::lead_hoeff}, i.e. $\tsig_2(q)=o(n^q\tsig_1(q))$, or equivalently, $\Sigma_2(q)=o(n^q\Sigma_1(q))$. Note that 
\begin{align*}
\sigma_1(l_1,l_2)=&\E[\varepsilon_1x_{1,l_1}\varepsilon_1x_{1,l_2}]=\sigma^2\Sigma_{l_1,l_2},\\
\sigma_2(l_1,l_2)=&\E[(\varepsilon_1x_{2,l_1}+\varepsilon_2x_{1,l_1})(\varepsilon_1x_{2,l_2}+\varepsilon_2x_{1,l_2})]=2\sigma^2\Sigma_{l_1,l_2}=2\sigma_1(l_1,l_2).
\end{align*}
Therefore, we have $\Sigma_2(q)=2^q\Sigma_1(q)$ and hence Assumption \ref{ass::lead_hoeff} holds.

\subsection{Component-wise independence testing with Kendall's \texorpdfstring{$\tau$}{LG}}
\label{sec:verify_kendall}
We first verify Assumption \ref{ass::lead_hoeff} by showing $\Sigma_2(q)=o(n^q\Sigma_1(q))$. Recall $\cL=\{l=(d,d'),1\le d< d'\le p\}$ in this example, and we denote $L=|\cL|=p(p-1)/2$. For $l=(d,d')$, as 
$$
h_l(X_{[2]})=\sgn(x_{1,d}-x_{2,d})\sgn(x_{1,d}-x_{2,d'}),
$$
we have under the null with component-wise independence,
\begin{align*}
h_l^{(1)}(X_1)=&\E[h_l(X_{[2]})|X_1]=(2F_d(x_{1,d})-1)(2F_{d'}(x_{1,d'})-1),\\
h_l^{(2)}(X_{[2]})=&h_l(X_{[2]})-h_l^{(1)}(X_1)-h_l^{(1)}(X_2),
\end{align*}
where we have used $F_d$ to denote the cdf of $x_{1,d}$. We assume $X_i$ has continuous distributions. Therefore, $U_d\deltaeq F_d(x_{1,d})\sim Unif(0,1)$, and they are independent under the null. We have
\begin{align*}
\sigma_1(l_1,l_2)=&\E[h_{l_1}^{(1)}(X_1)h_{l_2}^{(1)}(X_1)]=\E[(2U_{d_1}-1)(2U_{d_1'}-1)(2U_{d_2}-1)(2U_{d_2'}-1)]\\
=&I(l_1=l_2)\cdot \E[(2U_{d_1}-1)^2]\E[(2U_{d_1'}-1)^2]=\frac{1}{9}\cdot I(l_1=l_2),\\
\sigma_2(l_1,l_2)=&\E[h_{l_1}(X_{[2]})h_{l_1}(X_{[2]})]-2\E[h_{l_1}^{(1)}(X_1)h_{l_2}^{(1)}(X_1)]\\
=&I(l_1=l_2)-\frac{2}{9}I(l_1=l_2)=\frac{7}{9}I(l_1=l_2).
\end{align*}
Therefore, we have $\Sigma_2(q)=7^q\Sigma_1(q)=(7/9)^qL$ and hence Assumption \ref{ass::lead_hoeff} holds for $h_l^{(1)}$.

Since the kernel is not degenerate, it remains to check  Assumption \ref{ass::asy_deg} (a) and (b1).

For (a), we have 
\begin{align*}
\sum_{l_1,l_2,l_3,l_4\in\cL}[\sigma_1(l_1,l_2)\sigma_1(l_3,l_4)\sigma_1(l_1,l_4)\sigma_1(l_2,l_3)]^{q/2}=\left(\frac{7}{9}\right)^{2q}L.
\end{align*}
It is indeed $o(\Sigma_2^2(q))$ as $p\rightarrow\infty$ and as $\Sigma_2^2(q)=(7/9)^{2q}L^2$.

To simplify notations, we denote $h_{l_h}^{(1)}(X_1)$ by $Z_{l_h}$. For (b1), we have
\begin{align*}
\cum(Z_{l_1},\ldots,Z_{l_4})=\E[Z_{l_1}Z_{l_2}Z_{l_3}Z_{l_4}]-\E[Z_{l_1}Z_{l_2}]\E[Z_{l_3}Z_{l_4}]-\E[Z_{l_1}Z_{l_3}]\E[Z_{l_2}Z_{l_4}]-\E[Z_{l_1}Z_{l_4}]\E[Z_{l_2}Z_{l_3}].
\end{align*}
It is straightforward to see that each term in RHS is bounded by a constant. Moreover, if there is at least one index $l_h$ such that $d_h$ or $d_h'$ is not equal to any other, then RHS is zero since $U_d$'s are independent. Therefore, for a term to be non-zero, we must have each $d_h$ and $d_h'$ appear at least twice in $\{l_1,l_2,l_3,l_4\}=\{d_1,d_1',\ldots,d_4,d_4'\}$, which implies that the summation $\sum_{l_1,\ldots,l_4\in\cL}=\sum_{1\le d_1<d_1'\le p,\ldots,1\le d_4<d_4'\le p}$ contains at most $O(p^4)$ terms. 
We conclude that
\begin{align*}
\sum_{l_1,l_2,l_3,l_4\in\cL}\cum({h_{l_1}^{(1)}(X_1),\ldots,h_{l_4}^{(1)}(X_1)})=O(p^4)=O(L^2)=O(\Sigma_2^2(q)),
\end{align*}
and (b1) holds.

\section{Technical Proofs of the Theorems}
\label{sec:tech}

\subsection{Proof of Lemma \ref{lem::hoeff}}

For the ease of notation, we define the unnomarlized version of $U_{n,q}$ and $U_{n,\balpha}$  with 
$$
D_{n,q}=\frac{n!}{(n-qr)!}U_{n,q};\quad D_{n,\balpha}=\frac{n!}{(n-c)!}U_{n,\balpha},
$$
where $c=\sum_t t\alpha_t$ is the order of $h_l^{\balpha}$.

For $I=(i_1,i_2,\ldots,i_{qr})\in P_{qr}(n)$, denote $I_c=(i_{(c-1)r+1},\ldots,i_{cr})\in P_r(n)$ for $c=1,\ldots,q$, so we may also write $I=(I_1,I_2,\ldots,I_q)$. We have
\begin{align*}
D_{n,q}=&\sum_{l\in\cL}\sum_{I\in P_{qr}(n)}(\otimes^q h_l)(X_I)\\
=&\sum_{l\in\cL}\sum_{I\in P_{qr}(n)}\prod_{c=1}^q h_l(X_{I_c})\\
=&\sum_{l\in\cL}\sum_{I\in P_{qr}(n)}\prod_{c=1}^q \left(\sum_{t=1}^r\sum_{J\subseteq I_c:|J|=t}h_l^{(t)}(X_J)\right)\\
=&\sum_{l\in\cL}\sum_{c=q}^{qr}\frac{(n-c)!}{(n-qr)!}\sum_{\balpha:\sum_t t\alpha_t=c}C_{\balpha}\sum_{I\in P_c(n)}h_l^{\balpha}(X_I)\\
=&\sum_{c=q}^{qr}\frac{(n-c)!}{(n-qr)!}\sum_{\balpha:\sum_t t\alpha_t=c} C_{\balpha}D_{n,\balpha},
\end{align*}
for some constants $C_{\balpha}$, 
where we have used the fact that
\begin{equation}
h_l(X_I)=\sum_{t=1}^r\sum_{J\subseteq I:|J|=t}h_l^{(t)}(X_J),
\label{eq:sumhajek}
\end{equation}
and combine the terms with the same kernel $h_l^{\balpha}$ (i.e. with the same $\balpha$) into their corresponding U-statistic. Therefore we conclude that
$$U_{n,q}=\sum_{c=q}^{qr}\sum_{\balpha:\sum_t t\alpha_t=c}C_{\balpha}U_{n,\balpha}.$$
Since the kernel $\tilde h_{l,\balpha}$ with order $\sum_t t\alpha_t$ less than $qs$ equals 0, we may also start the index in the summation from $c=qs$ instead of $c=q$.

Note that $C_{\balpha}$ can be calculated explicitly in general. We only derive it for $\balpha=e_t(r)$ required for the leading term. In \eqref{eq:sumhajek}, the decomposition of $h_l$ contains $\binom{r}{t}$ terms in $h_l^{(t)}$. Therefore, the total number of terms in $\otimes^q h_l^{(t)}$ equals $P^n_{n-qr}\binom{r}{t}^q$. Therefore, for $\balpha=e_t(r)$, we have $P^n_{n-qr}\binom{r}{t}^q=\frac{(n-c)!}{(n-qr)!}P^n_c C_{\balpha}$, by matching the number of terms $\otimes^q h_l^{(t)}$, which implies that $C_{\balpha}=\binom{r}{t}^q$.

\subsection{Proof of Lemma \ref{lem::lead_hoeff}}
We analyze the asymptotic variance of $U_{n,\balpha}$.

For $\balpha=(\alpha_1,\ldots,\alpha_r)$, we define $$\tilde\sigma_{\balpha}(l_1,l_2)=\E[\tilde h_{l_1}^{\balpha}(X_I)\tilde h_{l_2}^{\balpha}(X_I)],$$
and $\tsig_{\balpha}=\sum_{l_1,l_2}\tilde\sigma_{\balpha}(l_1,l_2)$. Note that by Cauchy–Schwarz inequality, we have 
\begin{align*}
\tsig_{\balpha}=\E\left[\left(\sum_{l\in\cL}\tilde h_l^{\balpha}(X_I)\right)^2\right]\le\E\left[\left(\sum_{l\in\cL} h_l^{\balpha}(X_I)\right)^2\right]=\sum_{l_1,l_2\in\cL}\prod_{t=s}^r\sigma_t^{\alpha_t}(l_1,l_2),
\end{align*}
as $\sum_{l\in\cL}\tilde h_l^{\balpha}$ is the symmetrized kernel of $\sum_{l\in\cL}h_l^{\balpha}$.
Moreover, applying H\"{o}lder's inequality, we have 
$$
\sum_{l_1,l_2\in\cL}\prod_{t=s}^r\sigma_t^{\alpha_t}(l_1,l_2)\le \prod_{t=s}^r\left(\sum_{l_1,l_2\in\cL}\sigma_t^{q}(l_1,l_2)\right)^{\alpha_t/q}=\prod_{t=s}^r\Sigma_t^{\alpha_t/q}(q).
$$
Recall that the Hoeffding decomposition can be written as
$$
U_{n,q}=\binom{r}{s}^q U_{n,q}^{(qs)}+\sum_{c=qs+1}^{qr}\sum_{\balpha:\sum_t t\alpha_t=c}C_{\balpha}U_{n,\balpha},
$$
Note that $\tilde h_l^{\balpha}$ is fully degenerate and orthogonal, which implies $\var(U_{n,\balpha})=O(n^{-c}\tsig_{\balpha})$ with $c=\sum_t\alpha_t$ and $\var(U_{n,q}^{(qs)})=O(n^{-qs}\tsig_s(q))$. Therefore, we have for any $\balpha$ such that $\sum_t t\alpha_t>qs$,
\begin{align*}
 &\var(U_{n,\balpha})=O(n^{-c}\tsig_{\balpha})\\
 =&O(n^{-c}\Sigma_{\balpha})=O\left(n^{-c}\prod_{t=s}^r\Sigma_t^{\alpha_t/q}(q)\right) \quad(\tsig_{\balpha}\le \Sigma_{\balpha}\text{ and H\"{o}lder})\\
 =&o\left(n^{-c}\prod_{t=s}^rn^{\alpha_t(t-s)}\Sigma_s^{\alpha_t/q}(q)\right)\\
 =&o(n^{-qs}\tsig_{s}(q))=o\big(\var(U_{n,q}^{(qs)})\big),    
\end{align*}
by Assumption \ref{ass::lead_hoeff}.
Therefore, we have all terms but $U_{n,q}^{(qs)}$ are negligible and complete the proof.

\subsection{Proof of Proposition \ref{thm::mainr_deg}}
The degeneracy of $h_l$ implies that $\otimes^q h_l$ is also fully degenerate, and we have $D_{n,q}=D_{n,q}^{(qr)}$. Note that with a slight abuse of notation, we define $D_{n,q}^{(qr)}=\binom{n}{qr}U_{n,q}^{(qr)}$ (instead of the coefficient equal to $\frac{n!}{(n-qr)!}$) in this section, so that
$$
D_{n,q}^{(qr)}=\sum_{l\in\cL}\sum_{1\le i_1<\cdots<i_{qr}\le n}H_{l,q}^{(r)}(X_{i_1},\ldots, X_{i_{qr}}).
$$
. By Wold's device, it suffices to show that for any $q\not=Q\in2\Z_+$ and $\alpha_1,\alpha_2\in\R$,
$$\alpha_1n^{-qr/2}\tsig^{-1/2}_r(q)D_{n,q}^{(qr)}+\alpha_2n^{-Qr/2}\tsig^{-1/2}_r(Q)D_{n,Q}^{(Qr)}\cod N(0,\alpha_1^2(qr)!+\alpha_2^2(Qr)!).$$

The degeneracy of kernel implies that LHS is a martingale, whose martingale difference is
\begin{align*}
 \xi_{n,i}=&\alpha_1 (qr)!n^{-qr/2}\tsig^{-1/2}_r(q)\sum_{l\in\cL}\sum_{1\le i_1<\cdots<i_{qr-1}< i}H_{l,q}^{(r)}(X_{i_1},\ldots, X_{i_{qr-1}},X_i)\\
&+\alpha_2(Qr)!n^{-Qr/2}\tsig^{-1/2}_r(Q)\sum_{l\in\cL}\sum_{1\le i_1<\cdots<i_{Qr-1}< i}H_{l,Q}^{(r)}(X_{i_1},\ldots, X_{i_{Qr-1}},X_i)\\
\deltaeq&\alpha_1(qr)!\xi_{n,i}^{(1)}+\alpha_2(Qr)!\xi_{n,i}^{(2)}.
\end{align*}
Note that $(qr)!$ and $(Qr)!$ in the first line appear due to the permutation across $qr$ and $Qr$ indices.
To apply the martingale CLT (Theorem 35.12 in Billingsley (2008)), we need to verify the following two conditions.\\
$\displaystyle(i)\quad\forall\varepsilon>0,\sum_{i=1}^n\mathbb{E}\left[{\xi}_{n,i}^{2} \textbf{1}\left\{\left|{\xi}_{n,i}\right|>\varepsilon\right\} \Big| \mathcal{F}_{i-1}\right] \stackrel{p}{\rightarrow} 0$.\\
$\displaystyle(ii)\quad V_n=\sum_{i=1}^{ n}\mathbb{E}\left[{\xi}_{n,i}^{2}|\F_{i-1}\right] \stackrel{p}{\rightarrow} \alpha_1^2(qr)!+\alpha_2^2(Qr)!$.

To prove (i), it suffices to show that
$$
\sum_{i=1}^n\E[\xi_{n,i}^{4}]\lesssim\sum_{i=1}^n\E[(\xi_{n,i}^{(1)})^{4}+(\xi_{n,i}^{(2)})^{4}]\rightarrow 0.
$$
Note that
\begin{align*}
 &\sum_{i=1}^n\E[(\xi_{n,i}^{(1)})^{4}]=n^{-2qr}\tsig_r^{-2}(q)\sum_{i=1}^n\E\left[\left(\sum_{1\le i_1<\cdots<i_{qr-1}< i}\sum_{l\in\cL}H_{l,q}^{(r)}(X_{i_1},\ldots, X_{i_{qr-1}},X_i)\right)^4\right]\\
 \lesssim&\sum_{i=1}^n n^{-2qr}\tsig_r^{-2}(q)n^{2(qr-1)}\E\left[\prod_{u=1}^4H_{l_u,q}^{(r)}(X_{i_1^{(u)}},\ldots,X_{i_{qr-1}^{(u)}},X_i)\right]\\
 \lesssim&n^{-1}\rightarrow 0,
\end{align*}
where the first inequality uses the fact that all subscripts have to appear at least twice to have nonzero expectation and the second inequality is due to Lemma \ref{lem::assb2}.  
Similar argument leads to $\sum_{i=1}^n\E[(\xi_{n,i}^{(2)})^{4}]\rightarrow 0.$

To prove (ii), we decompose the summation into two parts,
\begin{align*}
&\sum_{i=qr}^n\E[(\xi_{n,i}^{(1)})^2|\F_{i-1}]\\
=&n^{-qr}\widetilde\Sigma_r^{-1}(q)\sum_{l_1,l_2\in\cL}\sum_{i=qr}^{n}\sum_{i_t^{(1)}<i,i_t^{(2)}<i}\E\left[H_{q,l_1}^{(r)}(X_{i_1^{(1)}},\ldots,X_{i_{qr-1}^{(1)}},X')H_{q,l_2}^{(r)}(X_{i_1^{(2)}},\ldots,X_{i_{qr-1}^{(2)}},X_i)\mid \F_{i-1}\right]\\
=&n^{-qr}\widetilde\Sigma_r^{-1}(q)\sum_{l_1,l_2\in\cL}\sum_{i=qr}^{n}\Bigg(\sum_{i_t^{(1)}<i,i_t^{(2)}<i}^{(1)}\E\left[H_{q,l_1}^{(r)}(X_{i_1^{(1)}},\ldots,X_{i_{qr-1}^{(1)}},X')H_{q,l_2}^{(r)}(X_{i_1^{(2)}},\ldots,X_{i_{qr-1}^{(2)}},X')\mid \F_{i-1}\right]\\
&+\sum_{i_t^{(1)}<i,i_t^{(2)}<i}^{(2)}\E\left[H_{q,l_1}^{(r)}(X_{i_1^{(1)}},\ldots,X_{i_{qr-1}^{(1)}},X')H_{q,l_2}^{(r)}(X_{i_1^{(2)}},\ldots,X_{i_{qr-1}^{(2)}},X')\mid\F_{i-1}\right]\Bigg)\\
\deltaeq&V_n^{(1)}+V_n^{(2)},
\end{align*}
where $\sum^{(1)}$ to be over the indices where $i_t^{(1)}$ and $i_t^{(2)}$ can be all paired up to a permutation, for $t=1,2,\ldots,qr-1$, and $\sum^{(2)}$ includes all others, and we have replace $X_i$ by a dummy variable $X'$ since we are taking its expectation (conditioning on others). Note that we have abbreviate $i_1^{(u)}<i_2^{(u)}<\cdots<i_{qr-1}^{(u)}<i$ as $i_t^{(u)}<i$ in the above expression of summation for $u=1,2$, to ease the notation.

We first show $V_n^{(1)}\cop [(qr)!]^{-1}$. 
Since 
\begin{align*}
\E[V_n^{(1)}]=&n^{-qr}\widetilde\Sigma_r^{-1}(q)\sum_{l_1,l_2\in\cL}\sum_{1\le i_1<\cdots<i_{qr}\le n}\E\left[H_{q,l_1}^{(r)}(X_{i_1},\ldots,X_{i_{qr-1}},X')H_{q,l_2}^{(r)}(X_{i_1},\ldots,X_{i_{qr-1}},X')\right]\\
\rightarrow& [(qr)!]^{-1},
\end{align*}
it suffices to show
$$
\E\left[\left(V_n^{(1)}\right)^2\right]\rightarrow [(qr)!]^{-2},
$$
which implies convergence in $L^2$ and hence the convergence in probability. In fact,
we have
\begin{align*}
\E[(V_n^{(1)})^2]=&n^{-2qr}\widetilde\Sigma_r^{-2}(q)\sum_{l_1,l_2,l_3,l_4\in\cL}\sum_{i,j=qr}^{n}\sum_{i_t<i}\sum_{j_t<j}\E\left[H_{q,l_1}^{(r)}(X_{i_1},\ldots,X_{i_{qr-1}},X')H_{q,l_2}^{(r)}(X_{i_1},\ldots,X_{i_{qr-1}},X')\right.\\
&~~~~~~~\left.H_{q,l_3}^{(r)}(X_{j_1},\ldots,X_{j_{qr-1}},X'')H_{q,l_4}^{(r)}(X_{j_1},\ldots,X_{j_{qr-1}},X'')\right].
\end{align*}
We decompose the summation into $\sum^{(a)}$ and $\sum^{(b)}$, where $\sum^{(a)}$ includes the term such that $\{i_t\}_{t=1}^{qr-1}$ and $\{j_t\}_{t=1}^{qr-1}$ have no intersection, and $\sum^{(b)}$ includes all others.
Note that
\begin{align*}
&n^{-2qr}\tsig_r^{-2}(q)\sum_{l_1,l_2,l_3,l_4\in\cL}\sum_{i,j=qr}^{n}\sum_{i_t<i,j_t<j}^{(a)}\E\left[H_{q,l_1}^{(r)}(X_{i_1},\ldots,X_{i_{qr-1}},X')H_{q,l_2}^{(r)}(X_{i_1},\ldots,X_{i_{qr-1}},X')\right.\\
&~~~~~~~~\left.H_{q,l_3}^{(r)}(X_{j_1},\ldots,X_{j_{qr-1}},X'')H_{q,l_4}^{(r)}(X_{j_1},\ldots,X_{j_{qr-1}},X'')\right]\\
\asymp &n^{-2qr}\tsig_r^{-2}(q)\sum_{i,j=qr}^{n}\binom{i}{qr-1}\binom{j}{qr-1}\sum_{l_1,l_2,l_3,l_4\in\cL}\E\big[H_{q,l_1}^{(r)}(X_{[qr]})H_{q,l_2}^{(r)}(X_{[qr]})\big]\E\big[H_{q,l_3}^{(r)}(X_{[qr]})H_{q,l_4}^{(r)}(X_{[qr]})\big]\\
\asymp&n^{-2qr}\tsig_r^{-2}(q)[(qr)!]^{-2}n^{2qr}\tsig_r^2(q)=[(qr)!]^{-2}.
\end{align*}
For $\sum^{(b)}$, we lose at least one degree of freedom due to the matching indices between $\{i_t\}$ and $\{j_t\}$. Therefore, by Lemma \ref{lem::assb2} (i), we have
\begin{align*}
&n^{-2qr}\widetilde\Sigma_r^{-2}(q)\sum_{l_1,l_2,l_3,l_4\in\cL}\sum_{i,j=qr}^{n}\sum_{i_t<i,j_t<j}^{(b)}\E\left[H_{q,l_1}^{(r)}(X_{i_1},\ldots,X_{i_{qr-1}},X')H_{q,l_2}^{(r)}(X_{i_1},\ldots,X_{i_{qr-1}},X')\right.\\
&~~~~~~~\left.H_{q,l_3}^{(r)}(X_{j_1},\ldots,X_{j_{qr-1}},X'')H_{q,l_4}^{(r)}(X_{j_1},\ldots,X_{j_{qr-1}},X'')\right]\\
=&O(\frac{1}{n})\tsig_r^{-2}(q)\cdot O(\tsig_r^2(q))\rightarrow 0..
\end{align*}
Therefore, we conclude $V_n^{(1)}$ converges to $[(qr)!]^{-2}$ in $L^2$.

For $V_n^{(2)}$, it suffices to show $\E[(V_n^{(2)})^2]\rightarrow 0$ to conclude its negligibly. Note that the summand in $V_n^{(2)}$ is nonzero only if any index in $\{i_t^{(u)},j_t^{(u)}\}$ appears at least twice. The summation over terms that has some index to appear more than twice is $O(1/n)$ due to the additional loss of order in $n$.
Therefore, we focus on terms whose indices all appear exacting twice, and the number of summands also has order $O(n^{2qr})$, similar to $\sum^{(a)}$ in the analysis of $V_n^{(1)}$. By Lemma \ref{lem::assb2}, we have
$$
V_n^{(2)}=n^{-2qr}\tsig_r^{-2}(q)\cdot O(n^{2qr})\cdot o(\tsig_r^2(q))\rightarrow 0.
$$
which implies $V_n^{(2)}$ converges to 0 in $L^2$.
Therefore, we conclude $\sum_{i=qr}^n\E[(\xi_{n,i}^{(1)})^2|\F_{i-1}]\cop[(qr)!]^{-2}$. Similarly, we have $\sum_{i=Qr}^n\E[(\xi_{n,i}^{(2)})^2|\F_{i-1}]\cop[(Qr)!]^{-2}$, and $\sum_{i=qr}^n\E[\xi_{n,i}^{(1)}\xi_{n,i}^{(2)}|\F_{i-1}]\cop 0$.
By putting all pieces together, we complete the proof.

\begin{lemma}
Given four groups of indices that are distinct within group. $\{i_1^{(u)},\ldots,i_{qr}^{(u)}\},\{j_1^{(u)},\ldots,j_{qr}^{(u)}\}$ for $u=1,2$.\\
(i) Under Assumption \ref{ass::asy_deg} (b1), we have
\begin{align*}
&\sum_{l_1,l_2,l_3,l_4\in\cL}\E\left[H_{q,l_1}^{(r)}(X_{i_1^{(1)}},\ldots,X_{i_{qr}^{(1)}})H_{q,l_2}^{(r)}(X_{i_1^{(2)}},\ldots,X_{i_{qr}^{(2)}})\right.\\
&~~~~\left.H_{q,l_3}^{(r)}(X_{j_1^{(1)}},\ldots,X_{j_{qr}^{(1)}})H_{q,l_4}^{(r)}(X_{j_1^{(2)}},\ldots,X_{j_{qr}^{(2)}})\right]\\
=&O(\tsig_r^2(q)).
\end{align*}
(ii) Moreover, suppose $\{i_1^{(1)},\ldots,i_{qr-1}^{(1)}\}$ are not identical to $\{i_1^{(2)},\ldots,i_{qr-1}^{(2)}\}$, and each element in $\{i_t^{(1)},i_t^{(2)},j_t^{(1)},j_t^{(2)}\}_{t=1}^{qr-1}$ appears exactly twice. Under Assumption \ref{ass::asy_deg} (a) and (b2), we have
\begin{align*}
&\sum_{l_1,l_2,l_3,l_4\in\cL}\E\left[H_{q,l_1}^{(r)}(X_{i_1^{(1)}},\ldots,X_{i_{qr-1}^{(1)}},X')H_{q,l_2}^{(r)}(X_{i_1^{(2)}},\ldots,X_{i_{qr-1}^{(2)}},X')\right.\\
&~~~~\left.H_{q,l_3}^{(r)}(X_{j_1^{(1)}},\ldots,X_{j_{qr-1}^{(1)}},X'')H_{q,l_4}^{(r)}(X_{j_1^{(2)}},\ldots,X_{j_{qr-1}^{(2)}},X'')\right]\\
=&o(\tsig_r^2(q)).
\end{align*}
\label{lem::assb2}
\end{lemma}
\begin{proof}[Proof of Lemma \ref{lem::assb2}]
For (i), by AM-GM inequality, since $H_{q,l}^{(r)}$ is the symmetrized kernel of $\otimes^q h_l^{(r)}$, we have
\begin{align*}
&\sum_{l_1,l_2,l_3,l_4\in\cL}\E\left[H_{q,l_1}^{(r)}(X_{i_1^{(1)}},\ldots,X_{i_{qr}^{(1)}})H_{q,l_2}^{(r)}(X_{i_1^{(2)}},\ldots,X_{i_{qr}^{(2)}})\right.\\
&~~~~\left.H_{q,l_3}^{(r)}(X_{j_1^{(1)}},\ldots,X_{j_{qr}^{(1)}})H_{q,l_4}^{(r)}(X_{j_1^{(2)}},\ldots,X_{j_{qr}^{(2)}})\right]\\
= &\E\left[\left(\sum_{l\in\cL} H_{q,l}^{(r)}(X_{i_1^{(1)}},\ldots,X_{i_{qr}^{(1)}})\right)\left(\sum_{l\in\cL}H_{q,l}^{(r)}(X_{i_1^{(2)}},\ldots,X_{i_{qr}^{(2)}})\right)\right.\\
&~~~~\left.\left(\sum_{l\in\cL}H_{q,l}^{(r)}(X_{j_1^{(1)}},\ldots,X_{j_{qr}^{(1)}})\right)\left(\sum_{l\in\cL}H_{q,l}^{(r)}(X_{j_1^{(2)}},\ldots,X_{j_{qr}^{(2)}})\right)\right]\\
\le & \E\left[\left(\sum_{l\in\cL} H_{q,l}^{(r)}(X_{i_1^{(1)}},\ldots,X_{i_{qr}^{(1)}})\right)^4\right]\quad\text{(H\"{o}lder's inequality)}\\
\le & \E\left[\left(\sum_{l\in\cL} \otimes^q h_l^{(r)}(X_1,\ldots,X_{qr})\right)^4\right]\quad\text{(AM-GM inequality, as $H_{q,l}^{(r)}$ is an average of $\otimes^qh_l$)}\\
= &\sum_{l_1,\ldots,l_4 \in\cL}\E[h_{l_1}(X_1,\ldots,X_r)\cdots h_{l_1}(X_{(q-1)r+1},\ldots,X_{qr})\cdots \\
& h_{l_4}(X_1,\ldots,X_r)\cdots h_{l_4}(X_{(q-1)r+1},\ldots,X_{qr})]
\end{align*}

The remaining proof of (i) is identical to Lemma 6.1 in \cite{zhang2021adaptive}. We include a self-contained proof here. We show a slightly stronger result that matches \cite{zhang2021adaptive} as follows.

Claim: Under Assumption \ref{ass::asy_deg} (b1), for any $i_1^{(u)},i_2^{(u)},...,  i_{qr}^{(u)}$ that are all distinct with  $u = 1,2,3,4$, we have
\begin{equation}
\begin{aligned}
\Big|\sum_{l_1,\ldots,l_4 \in\cL}\E[&h_{l_1}(X_{i_1^{(1)}},\ldots,X_{i_r^{(1)}})\cdots h_{l_1}(X_{i_{(q-1)r+1}^{(1)}},\ldots,X_{i_{qr}^{(1)}})\cdots \\
& h_{l_4}(X_{i_1^{(4)}},\ldots,X_{i_r^{(4)}})\cdots h_{l_4}(X_{i_{(q-1)r+1}^{(4)}},\ldots,X_{i_{qr}^{(4)}})]\Big| \lesssim \tsig_r^{2}(q).    
\end{aligned}
\label{lem_cum_more}
\end{equation}
Proof of Claim: Define $I_{t}^{(u)}=(i_{(r-1)t+1}^{(u)},\ldots,i_{(rt)}^{(u)})$.
By H\"older's Inequality, we have
\begin{align*}
&\left|\sum_{l_1,...,l_4\in\cL}\E[h_{l_1}(X_{I_1^{(1)}})\cdots h_{l_1}(X_{I_q^{(1)}}))\cdots h_{l_4}(X_{I_1^{(4)}})\cdots h_{l_4}(X_{I_q^{(4)}})]\right| \\
= &\left|\E\left[\prod_{u=1}^4\sum_{l_u\in\cL}[h_{l_u}(X_{I_1^{(u)}})\cdots h_{l_u}(X_{I_q^{(u)}})\right]\right|\\ 
\le &\prod_{u=1}^4\left\{\E\left[\left(\sum_{l_u\in\cL}[h_{l_u}(X_{I_1^{(u)}})\cdots h_{l_u}(X_{I_q^{(u)}})\right)^4\right]\right\}^{1/4}.
\end{align*}
Note that
\begin{align*}
&\E\left[\left(\sum_{l\in\cL}[h_{l}(X_1,\ldots,X_r)\cdots h_{l}(X_{(q-1)r+1},X_{qr})\right)^4\right]\\
=&\sum_{l_1,...,l_4\in\cL}\E^q[h_{l_1}(X_1,\ldots,X_r)\cdots h_{l_4}(X_1,\ldots,X_r)]\\
\lesssim&\sum_{l_1,...,l_4\in\cL}\E^q[h_{l_1}(X_1,\ldots,X_r)h_{l_2}(X_1,\ldots,X_r)]\E^q[h_{l_3}(X_1,\ldots,X_r)h_{l_4}(X_1,\ldots,X_r)]\\
&+\sum_{l_1,...,l_4\in\cL}|\cum^q(h_{l_1}(X_1,\ldots,X_r),\ldots,h_{l_4}(X_1,\ldots,X_r))|\\ 
\lesssim&\Sigma_r^2(q)\asymp\tsig_r^2(q).
\end{align*}
where the second last line in the above inequalities is due to the CR inequality and the definition of joint cumulants. We therefore conclude (i).

For (ii), to ease the notation, we write LHS as 
\begin{align*}
&\sum_{l_1,l_2,l_3,l_4\in\cL}\E\left[H_{q,l_1}^{(r)}(X_{i_1},\ldots,X_{i_{qr}})H_{q,l_2}^{(r)}(X_{i_{qr+1}},\ldots,X_{i_{2qr}})\right.\\
&~~~~\left.H_{q,l_3}^{(r)}(X_{i_{2qr+1}},\ldots,X_{i_{3qr}})H_{q,l_4}^{(r)}(X_{i_{3qr+1}},\ldots,X_{i_{4q}})\right],  
\end{align*}
with the constraint that $i_{qr}=i_{2qr}$ and $i_{3qr}=i_{4qr}$.
Since $H_{q,l}^{(r)}$ is the average of $\otimes^q h_l^{(r)}$ over different arguments, we may decompose the expectation in LHS as a linear combination of finite many terms with form
\begin{align*}
 [\otimes^q h_{l_1}^{(r)}(X_{I_1})][\otimes^q h_{l_2}^{(r)}(X_{I_2})][\otimes^q h_{l_3}^{(r)}(X_{I_3})][\otimes^q h_{l_4}^{(r)}(X_{I_4})],
\end{align*}
where $I_1,I_2,I_3,I_4$ are permutations of $\{i_t\}_{t=1}^{qr},\{i_t\}_{t=qr+1}^{2qr},\{i_t\}_{t=2qr+1}^{3qr},\{i_t\}_{t=3qr+1}^{4qr}$, respectively. We further write $I_u=(I_{u1},\ldots,I_{uq})$ where $I_{uc}$ is an index set containing $r$ elements, for $u=1,\ldots,4;\,c=1,\ldots,q$. Since each element in $\{I_1,\ldots,I_4\}$ appears exactly twice, we consider the following two cases.
\begin{itemize}
    \item The index set $I_{uc}$ pairs with each other. In this case, the constraint that $I_1\not=I_2$ and $i_{qr}=i_{2qr}, i_{3qr}=i_{4qr}$ implies that there exist $c_1,c_2,c_3,c_4$ such that $I_{1,c_1}=I_{3,c_3},I_{1,c_1}=I_{4,c_4}$, or $I_{1,c_1}=I_{4,c_4},I_{2,c_2}=I_{3,c_3}$. Moreover, the assumption on the indices implies that $I_{1q}=I_{2q}$ and $I_{3q}=I_{4q}$. Assumption \ref{ass::asy_deg} (a) implies that the term is $o_p(\tsig_r^2(q))$ by H\"older's inequality. The arguments are essentially the same as the proof of Lemma 6.3 in the supplemental material of \cite{zhang2021adaptive}. Therefore we only sketch the outline of the proof here for illustration, and we refer the readers to the supplemental material of \cite{zhang2021adaptive} for a complete proof.
    \item Suppose the index sets $\{I_{uc}\}$ are not all paired with each other. We then decompose the expectation using the cumulants, and write it as the summation over the products of cumulants. We use the cumulant decomposition to write the term as many products of cumulants and expectation, with form
    \begin{equation}
        \sum_{l_1,l_2,l_3,l_4\in\cL}\prod_{k=1}^K\cum(h_{u_1^{(k)}}(I_{u_1^{(k)},c_1^{(k)}}),\ldots,h_{u_{m_k}^{(k)}}(X_{I_{u_{m_k}^{(k)},c_{m_k}^{(k)}}})),
        \label{eq:B1iicumprod}
    \end{equation}
    where $(B_1,B_2,\ldots,B_K)$ is a partition of $(I_1,\ldots,I_4)$ with $B_k=(I_{u_1^{(k)},c_1^{(k)}},\ldots,I_{u_{m_k}^{(k)},c_{m_k}^{(k)}})$, we have assumed that the blocks are minimal in the sense that $B_1,\ldots,B_K$ have no overlap with each other. Otherwise, we may further decompose some blocks into the product of smaller blocks due to the independence. Since $I_{uc}$ are not all paired, there exist at least one block $B_{k_0}$ that contains $I_{uc}$ for at least three different $u$ in \{1,2,3,4\}, which we denote as
    \begin{equation}
           \cum^{q/m_k}(h_{l_1}(X_{i_1},\ldots,X_{i_r}),\ldots,h_{l_u}(X_{i_{(c-1)r+1}},\ldots,X_{i_{cr}}),
    \label{eq:B1iicum} 
    \end{equation}
    with $u=3$ or 4 and $c=m_k$, for simplicity.
    Using H\"older's inequality, we can bound (the absolute value of) \eqref{eq:B1iicumprod} by
    $$
     \prod_{k=1}^K\left(\sum_{l_1,l_2,l_3,l_4\in\cL}\cum^{q/m_k}(h_{u_1^{(k)}}(I_{u_1^{(k)},c_1^{(k)}}),\ldots,h_{u_{m_k}^{(k)}}(X_{I_{u_{m_k}^{(k)},c_{m_k}^{(k)}}}))\right)^{m_k/q}.
    $$
    Each term is $O(\tsig_r^{m_k/2q})$, so the product is $O(\tsig_r^{2})$ as $\sum_{k=1}^Km_k=4q$. However, note that the term in RHS associated with \eqref{eq:B1iicum}, i.e.
    $$\cum^{q/c}(h_{l_1}(X_{i_1},\ldots,X_{i_r}),\ldots,h_{l_u}(X_{i_{(c-1)r+1}},\ldots,X_{i_{cr}}),$$
    which is $o(\tsig_r^{c/2q}(q))$ by Assumption \ref{ass::asy_deg} (b2),
  This implies \eqref{eq:B1iicumprod} is $o_p(\tsig_r^2(q))$ and completes the proof.  
    
\end{itemize}

\end{proof}

\begin{remark}
Note that we have to make a stronger assumption (b2) in addition to (b1), due to the structure for kernels with order $r>2$. In fact, for $r=1$, Lemma \ref{lem::assb2} (ii) is implied by (b1), which implies the martingale CLT.
However, for general $r$, a direct extended assumption like (b1) does not imply the lemma trivially, and we have to impose a stronger assumption (b2).
\end{remark}

\subsection{Proof of Theorem \ref{thm::power}}
To simplify the notations, we define $$U_{n,q,l}=(P^n_{qr})^{-1}\sum_{1\le i_1,i_2,\cdots,i_{qr}\le n}(\otimes^q h_l)(X_1,\ldots,X_{qr}),$$ 
so that $\displaystyle U_{n,q}=\sum_{l\in\cL}U_{n,q,l}$, and
$$
U_{n,\balpha,l}=(P^n_{m})^{-1}\sum_{1\le i_1,i_2,\cdots,i_m\le n}h_l^{\balpha}(X_1,\ldots,X_m), 
$$
where $m$ is the order of $h_l^{\balpha}$, i.e., $m=\sum_t t\alpha_t$, so that
$\displaystyle U_{n,\balpha}=\sum_{l\in\cL}U_{n,\balpha,l}$.

Recall $\tilde U_{n,q}$ is the proposed U-statistic associated with the centered kernel $\tilde h_l$. We also define the counterpart of $U_{n,q,l}$ and $U_{n,\balpha,l}$ associated with $\tilde h_l$ as $\tilde U_{n,q,l}$ and $\tilde U_{n,\balpha,l}$.

Note that 
\begin{align*}
U_{n,q}=&\|\Theta\|_q^q+\tilde U_{n,q}+\sum_{c=1}^{q-1}\binom{q}{c}\sum_{l\in\cL} \theta_l^{q-c} \tilde U_{n,c,l}.\\
\end{align*}
Note that $\tilde U_{n,c,l}$ is the U-statistic associated with $\otimes^c \tilde h_l$. We may apply the Hoeffding decomposition in Lemma \ref{lem::hoeff} to decompose $\tilde U_{n,c,l}$ into $\tilde U_{n,\balpha,l}$ with all valid $\balpha$ such that $\sum_{t=1}^r\alpha_t=c$. In particular, we have $\alpha_t=0$ for any $t<s$ as $h_l$ has order of degeneracy $s$.

By calculating the variance along with apply the H\"older's inequality, we have for any $1 \le c \le q-1$,
\begin{equation}
\begin{aligned}
 &\var\left(\sum_{l\in\cL} \theta_l^{q-c} \tilde U_{n,\balpha,l}\right)\\
\lesssim& \sum_{l_1,l_2\in\cL}\theta_{l_1}^{q-c}\theta_{l_2}^{q-c}n^{-\sum_{t=s}^r t\alpha_t}\cov(\tilde h_{l_1}^{\balpha},\,\tilde h_{l_2}^{\balpha})\\
\lesssim& n^{-\sum_{t=s}^r t\alpha_t}\var\left(\sum_{l\in\cL}\theta_l^{q-c}\tilde h_l^{\balpha}\right)\\
\le& n^{-\sum_{t=s}^r t\alpha_t}\var\left(\sum_{l\in\cL}\theta_l^{q-c} h_l^{\balpha}\right) \quad\text{(AM-GM inequality)}\\
=&n^{-\sum_{t=s}^r t\alpha_t}\sum_{l_1,l_2\in\cL}\theta_{l_1}^{q-c}\theta_{l_2}^{q-c}\prod_{t=s}^r\sigma_t^{\alpha_t}(l_1,l_2) \\
\le & n^{-\sum_{t=s}^r t\alpha_t}\left(\sum_{l_1,l_2\in\cL}\theta_{l_1}^q\theta_{l_2}^q\right)^{(q-c)/q}\prod_{t=s}^r\left(\sum_{l_1,l_2\in\cL}\sigma_t^q(l_1,l_2)\right)^{\alpha_t/q} \quad (\text{H\"{o}lder})\\
 =&n^{-\sum_{t=s}^r t\alpha_t}\|\Theta\|_q^{2(q-c)}\prod_{t=s}^r\Sigma_t^{\alpha_t/q}(q).
\end{aligned}
\label{eq::var_cross}
\end{equation}
where the first H\"older's inequality is applied in the same manner as the proof of Lemma \ref{lem::hoeff}.
if $\alpha_s<c$ so that $\alpha_t>0$ for some $t>s$, Assumption \ref{ass::lead_hoeff} implies that RHS is
$$
n^{-\sum_{t=s}^r t\alpha_t}\|\Theta\|_q^{2(q-c)}\prod_{t=s}^ro\big([n^{q(t-s)}\Sigma_s(q)]^{\alpha_t/q}\big)=o\big(n^{-cs}\Sigma_s^{c/q}(q)\|\Theta\|_q^{2(q-c)}\big).
$$
On the other hand, if $\alpha_s=c$ so that it is the only non-zero component in $\balpha$, we have RHS equals
$n^{-cs}\|\Theta\|_q^{2(q-c)}\Sigma_s^{c/q}(q)$.

In summary, we always have (recall that $\Sigma_s(q)$ and $\tsig_s(q)$ have the same order)
$$\var\left(\sum_{l\in\cL} \theta_l^{q-c} \tilde U_{n,\balpha,l}\right)\lesssim \left(\|\Theta\|_q^{2q}\right)^{(q-c)/q}\left(n^{-qs}\tsig_s(q)\right)^{c/q}.$$
which is an geometric average of $\|\Theta\|_q^{2q}$ and $n^{-qs}\tsig_s(q)$. Therefore, if $\gamma_{n,q}\rightarrow 0$, $\var^{1/2}(\tilde U_{n,q})\propto n^{-qs/2}\tsig_s^{1/2}(q)$ dominates $\|\Theta\|_q^q$ and all $\sum_{l\in\cL} \theta_l^{q-c} \tilde U_{n,\balpha,l}$, which implies $U_{n,q}$ has the same asymptotic distribution as $\tilde U_{n,q}$.

On the other hand, if $\gamma_{n,q}\rightarrow \infty$, $\|\Theta\|_q^q$ is the leading term of $U_{n,q}$ which diverges to $\infty$.

As regards the local alternative where $\gamma_{n,q}\rightarrow \gamma\in (0,\infty)$, it suffices to show the cross-product terms are still dominated by $\|\Theta\|_q^q$ and $\tilde U_{n,q}$, which have the same order in this case, under Assumption \ref{ass::power_cross}. Note that we have proved that if  $\alpha_t>0$ for some $t>s$, then 
$$
\var\left(\sum_{l\in\cL} \theta_l^{q-c} \tilde U_{n,\balpha,l}\right)=o\big(n^{-cs}\tsig_s^{c/q}(q)\|\Theta\|_q^{2(q-c)}\big)=o\big(\var(\tilde U_{n,q})\big).
$$
Therefore, it remains to show that under Assumption \ref{ass::power_cross}, the result also holds for $\balpha=ce_s(r)$. In this case, we have shown in the third line of (\ref{eq::var_cross}) that
\begin{align*}
\var\left(\sum_{l\in\cL} \theta_l^{q-c} \tilde U_{n,\balpha,l}\right)
\lesssim n^{-cs}\sum_{l_1,l_2\in\cL}\theta_{l_1}^{q-c}\theta_{l_2}^{q-c}\tsig_s^c(l_1,l_2),
\end{align*}
which is $o\big(n^{-qs}\tsig_s(q)\big)$ by Assumption \ref{ass::power_cross}. Therefore, we have the desired asymptotic result under the local alternatives.

\subsection{Proof of Proposition \ref{prop::optim_q}}
Define $f(q)=[(qs)!]^{1/2q}(\sqrt{N}R/d)^{1/q}$, which removes all the terms irrelevant to $q$ in $\delta(q)$.

Suppose $d\ge\sqrt{N}R$. Then $\sqrt{N}R/d$ is an increasing function in $q$. As $[(qs)!]^{1/2q}$ is also increasing in $q$, we have $f(q)$ achieves its minimum at $q=2$ for all $q \in 2\mathbb{Z_+}$.

When $d<\sqrt{N}R$, we follow the idea as the proof of Proposition 2.3 in \cite{he2021asymptotically}. Define $D=\sqrt{N}R/d>1$ to simplify the notations.
Note that we have
\begin{align*}
    \frac{f(q+1)}{f(q)}=&\frac{\{[(q+1)s]!\}^{1/(2q+1)}D^{1/(q+1)}}{[(qs)!]^{1/2q}D^{1/q}}\\
    =&\left(\frac{\left(P^{(q+1)s}_s\right)^q}{(qs)!D^2}\right)^{\frac{1}{2q(q+1)}}.
\end{align*}
As a result, we have $f(q+1)>f(q)$ if and only if $\displaystyle g(q)\deltaeq \left(P^{(q+1)s}_s\right)^q/(qs)!>D^2$. It is not difficult to see that $g(q)$ is increasing in $q$, as
$$
\frac{g(q+1)}{g(q)}= \left(\frac{P^{(q+2)s}_s}{P^{(q+1)s}_s}\right)^{q+1}>1.
$$
This implies that if $f(q+1)\ge f(q)$ then we must have $f(q+2)>f(q+1)$, and the optimal $q$ is the smallest integer $q_0$ such that $g(q_0)\ge D^2$ by noting that $g(q)<D^2$ for any $q<q_0$ implies $f(q)$ is decreasing for $q<q_0$, and $f(q)$ is strictly increasing for $q\ge q_0$. Since we have shown $g(q)$ is strictly increasing, we conclude that $q_0$ is increasing in $D$.

\end{document}